\documentclass[reqno,12pt]{amsart}
\usepackage{amssymb,amsmath,amsthm,amstext,amsfonts,euscript}
\usepackage{graphicx}
\usepackage{enumitem} 

\makeatletter \@addtoreset{equation}{section} \makeatother

\theoremstyle{plain}
\newtheorem{thm}{Theorem}[section]
\newtheorem{prop}[thm]{Proposition}
\newtheorem{lemma}[thm]{Lemma}

\newtheorem{cor}[thm]{Corollary}
\theoremstyle{definition} \theoremstyle{remark}
\newtheorem{rmk}[thm]{Remark}
\newtheorem{defn}[thm]{Definition}
\newenvironment{pfof}[1]{\vspace{1ex}\noindent{\em Proof of
#1}.}{\hfill\qed\vspace{1ex}}

\newcommand{\wt}{\widetilde}

\newcommand{\diam}{\operatorname{diam}}
\newcommand{\dist}{\operatorname{dist}}
\newcommand{\supp}{\operatorname{supp}}
\newcommand{\Cov}{\operatorname{Cov}}

\newcommand{\overbar}[1]{\mkern 1.5mu\overline{\mkern-1.5mu#1\mkern-1.5mu}\mkern 1.5mu} 
\newcommand{\overbarr}[1]{\mkern 3.5mu\overline{\mkern-3.5mu#1\mkern-0.5mu}\mkern 0.5mu} 
\newcommand{\bX}{{\overbarr X}} 
\newcommand{\bY}{{\overbar Y}} 
\newcommand{\bF}{\overbarr F} 

\newcommand{\R}{{\mathbb R}}
\newcommand{\N}{{\mathbb N}}
\newcommand{\C}{{\mathbb C}}
\newcommand{\Z}{{\mathbb Z}}
\renewcommand{\P}{{\mathbb{P}}}

\newcommand{\cC}{{\mathcal C}}
\newcommand{\cD}{{\mathcal D}}
\newcommand{\cS}{{\mathcal S}}
\newcommand{\cW}{{\mathcal{W}}}

 \newcommand{\eps}{{\epsilon}}

 \newcommand{\Cl}{\operatorname{cl}}
\newcommand{\Int}{\operatorname{Int}}
\newcommand{\Lip}{\operatorname{Lip}}
\newcommand{\Leb}{\operatorname{Leb}}

 \def \fX {{\mathfrak X}}

\begin{document}

\title{Mixing properties and statistical limit theorems for singular hyperbolic 
  flows \\ without a smooth stable foliation}

\author{V. Ara\'ujo and I. Melbourne}

\address{Vitor Ara\'ujo,
 Departamento de Matem\'atica, Universidade Federal da Bahia\\
Av. Ademar de Barros s/n, 40170-110 Salvador, Brazil.}
\email{vitor.d.araujo@ufba.br,
 www.sd.mat.ufba.br/$\sim$vitor.d.araujo}

\address{Ian Melbourne,
Mathematics Institute, University of Warwick, Coventry CV4 7AL, UK}
\email{i.melbourne@warwick.ac.uk}

\thanks{
V.A.
  is partially supported by CNPq,
  PRONEX-Dyn.Syst. and FAPESB (Brazil).  
I.M. is partially supported by 
	by a European Advanced Grant StochExtHomog (ERC AdG 320977) and by CNPq
  (Brazil) through PVE grant number 313759/2014-6.  
We are grateful to the referee for very helpful suggestions which improved the readability of the paper.
}

  \begin{abstract}
Over the last 10 years or so, advanced statistical properties, including exponential decay of correlations, have been established for certain classes of singular hyperbolic flows in three dimensions.  The results apply in particular to the classical Lorenz attractor.  However, many of the proofs rely heavily on the smoothness of the stable foliation for the flow.  

In this paper, we show that many statistical properties hold for singular hyperbolic flows with no smoothness assumption on the stable foliation.  These properties include existence of SRB measures, central limit theorems and associated invariance principles, as well as results on mixing and rates of mixing.  The properties hold equally for singular hyperbolic flows in higher dimensions provided the center-unstable subspaces are two-dimensional.
\end{abstract}

 \date{3 December 2017.  Revised 2 March 2019.}

\maketitle

 \section{Introduction}  \label{sec:intro}

Singular hyperbolicity is a far-reaching generalization of Smale's notion of Axiom~A~\cite{Smale67} that allows for the inclusion of equilibria (also known as singular points or steady-states) and incorporates the classical Lorenz attractor~\cite{Lorenz63} as well as the geometric Lorenz attractors of~\cite{AfraimovicBykovSilnikov77,GuckenWilliams79}.
For three-dimensional flows, singular hyperbolic attractors are precisely the ones that are robustly transitive,
and they reduce to Axiom A attractors when there are no equilibria~\cite{MoralesPacificoPujals04}.

For the classical Lorenz attractor, strong statistical properties such as exponential decay of correlations, the central limit theorem (CLT), and associated invariance principles have been established in~\cite{AraujoM16,AraujoM17,AMV15,HM07}.
However the proofs rely heavily on the existence of a smooth stable foliation for the flow.
Various issues regarding the existence and smoothness of the stable foliation are clarified in~\cite{AraujoM17}; a topological foliation always exists, and an analytic proof of smoothness of the foliation for the classical Lorenz attractor (and nearby attractors) is given in~\cite{AraujoM17,AMV15}.

Even for three-dimensional flows, the stable foliation for a singular hyperbolic attractor need not be better than H\"older.
In this paper, we consider statistical properties for singular hyperbolic attractors that do not have a smooth stable foliation.   We do not restrict to three-dimensional flows, but our main results assume that the stable foliation has codimension two.  

\subsubsection*{Main results}
For codimension two singular hyperbolic attracting sets, we prove that the stable foliation is at least H\"older continuous (Theorem~\ref{thm:holder}), and using Pesin theory~\cite{BarreiraPesin} we deduce that the stable holonomies are absolutely continuous with H\"older Jacobians (Theorem~\ref{thm:H}).  As a consequence of this, we obtain that the 
stable holonomies for the associated Poincar\'e map are $C^{1+\eps}$
(since they are one-dimensional with H\"older Jacobians).
This extends results of~\cite{AfraimovicBykovSilnikov77,Robinson81,Shashkov94} who obtain a $C^1$ result for geometric Lorenz attractors (see the discussion after equation~(6) in~\cite{OT17}).
Quotienting out by the stable foliation, we obtain a $C^{1+\eps}$
one-dimensional expanding map.
We can now proceed following~\cite{APPV09} to obtain a spectral decomposition for the singular hyperbolic attracting set (Theorem~\ref{thm:spectral}).

To study statistical properties, we focus attention on  the transitive components of a singular  hyperbolic attracting set; these are called singular hyperbolic attractors.  In the Axiom~A case, the CLT and associated invariance principles are well-known~\cite{DenkerPhilipp84,MT04,Ratner73} and we extend these results to general (codimension two) singular hyperbolic attractors.
In particular, as described in Section~\ref{sec:statflow}, the (functional) CLT and related results follow from~\cite{BMprep} using the results in this paper.
Moreover, many strong limit laws are obtained for the associated Poincar\'e maps in Theorem~\ref{thm:statf}.

Mixing and rates of mixing for Axiom~A attractors are less well-understood even today, but an open and dense set of Axiom~A attractors have superpolynomial decay of correlations~\cite{Dolgopyat98a,FMT07}.   Theorem~\ref{thm:super} shows that the same result holds for singular hyperbolic attractors.  
As a consequence, Corollary~\ref{cor:super}, we obtain the CLT and almost sure invariance principle for the time-one map of the flow for this open and dense set of singular hyperbolic attractors and sufficiently smooth observables.
(We note that such results are much more delicate for time-one maps than for the flow and for Poincar\'e maps.)

In fact, for singular hyperbolic attractors containing at least one equilibrium and with a smooth stable foliation, mixing~\cite{LMP05}, superpolynomial decay of correlations~\cite{AMV15}, and exponential decay of correlations~\cite{AraujoM16} are automatic subject to a certain indecomposability condition (locally eventually onto).  Theorem~\ref{thm:mix} yields a similar result on automatic mixing when there is not a smooth stable foliation.  However, automatic rates of mixing, or any results on exponential decay of correlations, seems beyond current techniques when the stable foliation is not smooth.

\subsubsection*{Example}

In a recent paper, Ovsyannikov \& Turaev~\cite{OT17} (see also previous work of~\cite{DumortierKokubuOka95}) give an analytic proof of singular hyperbolic attractors in the extended Lorenz model
\[
\dot x = y, \quad \dot y=-\lambda y+\gamma x(1-z)-\delta x^3,
\quad \dot z=-\alpha z+\beta x^2.
\]
The attractors contain precisely one equilibrium, namely the origin, and are of geometric Lorenz type~\cite{AfraimovicBykovSilnikov77,GuckenWilliams79}.
The eigenvalues of the linearized equations at the equilibrium are close to $-1$, $-1$ and $1$ (up to a scaling) for the parameters considered in~\cite{OT17}, so the standard $q$-bunching condition~\cite{AraujoM17,HPS77} guaranteeing a $C^q$ stable foliation holds only for $q$ close to zero.  In this situation it is anticipated that the foliation fails to be $C^1$ except in pathological cases.
In particular, previous results on statistical properties for singular hyperbolic flows do not apply.  However, the results in the present paper do not require a smooth foliation.  It follows that the attractors in~\cite{OT17} satisfy the 
statistical limit laws described in this paper.
Moreover, there is an open set $\mathcal{U}$ within the space of  $C^2$ flows on $\R^3$, containing the extended Lorenz examples of~\cite{OT17}, that satisfy these statistical limit laws.  In addition, an open and dense set of flows in $\mathcal{U}$ have  superpolynomial decay of correlations.

\subsubsection*{Spectral decompositions}
Whereas the results on statistical properties for singular hyperbolic flows in this paper are completely new, we note that there are existing results on spectral decompositions~\cite{APPV09,LeplaideurYang}.
The decomposition in~\cite{APPV09} is for three-dimensional flows and our method extends~\cite{APPV09} in the more general codimension two situation.  The method in~\cite{LeplaideurYang} works directly with the flow and does not require the codimension two restriction.  
However~\cite{APPV09,LeplaideurYang} both make liberal use of Pesin theory, including results that seem currently unavailable in the literature.  The main issue, as clarified in~\cite{AraujoM17}, is that {\em a priori} the stable lamination over a partially hyperbolic attracting set~$\Lambda$ need not cover a neighborhood of $\Lambda$.  The stable bundle extends to an invariant contracting bundle over a neighborhood $U\supset\Lambda$ and this integrates to a topological foliation of~$U$.  However, the complementary center-unstable bundle does not extend invariantly, so the resulting extended splitting is not invariant.  This means that the application of Pesin theory in~\cite{APPV09,LeplaideurYang} is inaccurate.  It is likely that the desired results hold (some aspects were extended to noninvariant splittings already in~\cite{AraujoM17}) but currently the arguments seem incomplete.

In this paper, we make the approach in~\cite{APPV09} completely rigorous by bypassing the issue of noninvariance of the extended splitting.  Theorem~\ref{thm:fol} below shows that {\em a posteriori} the stable bundle restricted to $\Lambda$ integrates to a topological foliation.  This relies heavily on the special structure associated to a codimension two singular hyperbolic attracting set and uses also the information about the extended bundle~\cite{AraujoM17}.  Consequently, we can work with the nonextended splitting which is invariant and Pesin theory applies.  Also, using~\cite{PSW97} we show that the foliation is H\"older which simplifies the arguments in~\cite{APPV09}.

\subsubsection*{Sectional hyperbolicity}
Finally, we remark on the restriction to singular hyperbolic attracting sets that are codimension two.  The natural setting in general is to consider {\em sectional} hyperbolic attracting sets~\cite{MetzgerMorales08} (in the codimension two case, sectional and singular hyperbolicity are the same).
The proof of Theorem~\ref{thm:fol} (specifically Proposition~\ref{prop:cross}) relies on the restriction to codimension two.  Nevertheless,
we expect that in the sectional hyperbolic setting, our results on the stable foliation should go through largely unchanged (after adapting various arguments to deal with the noninvariant splitting).  However, the quotient map is higher-dimensional and so Pesin theory only gives a H\"older Jacobian; the map itself is no better than H\"older.  Hence the arguments in Section~\ref{sec:stat} and~\ref{sec:SH} on spectral decompositions and statistical properties break down; this remains the subject of future work.

\subsubsection*{}

The remainder of the paper is organized as follows.
In Section~\ref{sec:PH}, we review background material on partially hyperbolic attracting sets and  singular hyperbolicity, and
recall results on stable foliations from~\cite{AraujoM16}.
In Section~\ref{sec:f}, we construct a global Poincar\'e map~$f$ associated to any partially hyperbolic attracting set, following (and modifying) the construction in~\cite{APPV09}.
Section~\ref{sec:UH} establishes that $f$ is uniformly hyperbolic (with singularities) when the attracting set is singular hyperbolic.  

In Section~\ref{sec:fol}, we show that the stable lamination over an attracting codimension two singular hyperbolic set is a topological foliation.  In Section~\ref{sec:regWs},
we establish H\"older regularity and absolute continuity of the stable foliation, and show that the stable holonomies have H\"older Jacobians.
Using this, we obtain a uniformly expanding
piecewise $C^{1+\eps}$ quotient map $\bar f$ in Section~\ref{sec:barf}.

Finally, in Sections~\ref{sec:stat} and~\ref{sec:SH}, we prove results on spectral decompositions, statistical limit laws, and rates of mixing, for $\bar f$, $f$, and the underlying flow.

\subsubsection*{Notation}
Let $(M,d)$ be a metric space and $\eta\in(0,1)$. Given $v:M\to\R$,
define ${\|v\|}_{C^\eta}=|v|_\infty+{|v|}_{C^\eta}$ where
${|v|}_{C^\eta}=\sup_{x\neq x'}|v(x)-v(x')|/d(x,x')^\eta$.  We say that $v$ is $C^\eta$ and
write
$v\in C^\eta(M)$ if ${\|v\|}_{C^\eta}<\infty$.

\section{Singular  hyperbolic attracting sets}
\label{sec:PH}

In this section, we define what is understood as a singular
hyperbolic attracting set.
Throughout this paper, we restrict mainly to the case where the center-unstable subspace is two-dimensional.

Let $M$ be a compact Riemannian manifold and $\fX^r(M)$,
$r>1$, be the set of $C^r$ vector fields on $M$.
Let $Z_t$ denote the flow generated by $G\in\fX^r(M)$.
Given a compact invariant set $\Lambda$ for $G\in
\fX^r(M)$, we say that $\Lambda$ is \emph{isolated} if
there exists an open set $U\supset \Lambda$ such that
$
\Lambda =\bigcap_{t\in\R}Z_t(U).
$
If $U$ can be chosen so that $Z_t(U)\subset U$ for
all $t>0$, then we say that $\Lambda$ is an \emph{attracting set}.

\begin{defn}
\label{def:PH}
Let $\Lambda$ be a compact invariant set for $G \in
\fX^r(M)$.  We say that $\Lambda$ is {\em partially
  hyperbolic} if the tangent bundle over $\Lambda$ can be
written as a continuous $DZ_t$-invariant sum
$$
T_\Lambda M=E^s\oplus E^{cu},
$$
where $d_s=\dim E^s_x\ge1$ and $d_{cu}=\dim E^{cu}_x=2$ for $x\in\Lambda$,
and there exist constants $C>0$, $\lambda\in(0,1)$ such that
for all $x \in \Lambda$, $t\ge0$, we have
\begin{itemize}
\item uniform contraction along $E^s$:
\begin{align}\label{eq:contract}
\|DZ_t | E^s_x\| \le C \lambda^t;
\end{align}

\item domination of the splitting:
\begin{align}\label{eq:domination}
\|DZ_t | E^s_x\| \cdot \|DZ_{-t} | E^{cu}_{Z_tx}\| \le C \lambda^t.
\end{align}
\end{itemize}
We refer to $E^s$ as the stable bundle and to
$E^{cu}$ as the center-unstable bundle.
A {\em partially hyperbolic attracting set} is a partially hyperbolic set that is also an attracting set.
\end{defn}

\begin{defn} \label{def:VE}
The center-unstable bundle $E^{cu}$ is \emph{volume
  expanding} if there exists $K,\theta>0$ such that
$|\det(DZ_t| E^{cu}_x)|\geq K  e^{\theta t}$ for all
$x\in \Lambda$, $t\geq 0$.
\end{defn}

If $\sigma\in M$ and $G(\sigma)=0$, then $\sigma$ is called
an {\em equilibrium}.  An invariant set is \emph{nontrivial}
if it is neither a periodic orbit nor an equilibrium.

\begin{defn} \label{def:singularset}
Let $\Lambda$ be a compact nontrivial invariant set for $G \in
{\fX}^r(M)$.  We say that $\Lambda$ is a
\emph{singular hyperbolic set} if all equilibria
in $\Lambda$ are hyperbolic, and $\Lambda$ is partially
hyperbolic with volume expanding center-unstable bundle. 
A singular hyperbolic set which is also an attracting set is
called a {\em singular hyperbolic attracting set}.
\end{defn}

\begin{rmk} \label{rmk:per}
A singular hyperbolic attracting set contains no isolated periodic orbits.
For such a periodic orbit would have to be a periodic sink, violating volume expansion.
\end{rmk}

A subset $\Lambda \subset M$ is \emph{transitive} if it has
a full dense orbit, that is, there exists $x\in \Lambda$ such that
$\Cl{\{Z_tx:t\ge0\}}=\Lambda= \Cl{\{Z_tx:t\le0\}}$.

\begin{defn}\label{def:attractor}
  A \emph{singular hyperbolic attractor} is a transitive
  singular hyperbolic attracting set.
\end{defn}

\begin{prop} \label{prop:like}
Suppose that $\Lambda$ is a singular hyperbolic attractor with $d_{cu}=2$, and
let $\sigma\in\Lambda$ be an equilibrium.  Then
$\sigma$ is
 \emph{Lorenz-like}.   That is, 
$DG(\sigma)|E^{cu}_\sigma$ has real eigenvalues $\lambda^s$, $\lambda^u$ satisfying
$-\lambda^u<\lambda^s<0<\lambda^u$.
\end{prop}

\begin{proof}
It follows from Definition~\ref{def:singularset} that $\sigma$ is a hyperbolic saddle and that at most two eigenvalues have positive real part.
If there is only one such eigenvalue $\lambda^u>0$ then the constraints on $\lambda^s$ follow from volume expansion.

Let $\gamma$ be the local stable manifold for $\sigma$.
It remains to rule out the case $\dim \gamma=\dim M-2$.
In this case, $T_p\gamma=E^s_p$ for all $p\in \gamma\cap\Lambda$ and in particular $G(p)\in E^s_p$.
Also, $G(p)\in E^{cu}_p$
(see for example~\cite[Lemma~6.1]{AraujoPacifico}),
so we deduce that 
$G(p)=0$ for all $p\in \gamma\cap\Lambda$ and hence that
$\gamma\cap\Lambda=\{\sigma\}$.

On the other hand, $\Lambda$ is transitive and nontrivial, so there exists $x\in\Lambda\setminus\{\sigma\}$ such that $\sigma\in\omega(x)$.
  By the local behavior of orbits near hyperbolic saddles, there exists
  $p\in
  (\gamma\setminus\{\sigma\})\cap\omega(x)\subset
  (\gamma\setminus\{\sigma\})\cap\Lambda$ which as we have seen is impossible.
\end{proof}

We end this section by recalling/extending some results from~\cite{AraujoM17}.
These results hold for general $d_{cu}\ge2$.

\begin{prop} \label{prop:Es} Let $\Lambda$ be a partially
  hyperbolic attracting set.  The stable bundle $E^s$ over
  $\Lambda$ extends to a continuous uniformly contracting
  $DZ_t$-invariant bundle $E^s$ over an open neighborhood of
  $\Lambda$.
\end{prop}

\begin{proof}  See~\cite[Proposition~3.2]{AraujoM17}.
\end{proof}

Let $\cD^k$ denote the $k$-dimensional open unit disk and
let $\mathrm{Emb}^r(\cD^k,M)$ denote the set of $C^r$
embeddings $\phi:\cD^k\to M$ endowed with the $C^r$ distance.

\begin{prop}\label{prop:Ws}
Let $\Lambda$ be a partially hyperbolic attracting set.
	There exists a positively invariant neighborhood $U_0$
	of $\Lambda$, and constants $C>0$, $\lambda\in(0,1)$, such
	that the following are true:

\vspace{1ex}
 \noindent(a)
For every point $x \in U_0$ there is a $C^r$ embedded $d_s$-dimensional disk
  $W^s_x\subset M$, with $x\in W^s_x$, such that
\begin{enumerate}
	\item $T_xW^s_x=E^s_x$.
\item $Z_t(W^s_x)\subset W^s_{Z_tx}$ for all $t\ge0$.
\item $d(Z_tx,Z_ty)\le C\lambda^t d(x,y)$ for all $y\in W^s_x$, $t\ge0$.
\end{enumerate}

\vspace{1ex}
\noindent(b) The disks $W^s_x$ depend continuously on $x$ in the $C^0$ topology: there is a continuous map $\gamma:U_0\to {\rm Emb}^0(\cD^{d_s},M)$ such that
$\gamma(x)(0)=x$ and $\gamma(x)(\cD^{d_s})=W^s_x$.
Moreover, there exists $L>0$ such that $\Lip\gamma(x)\le L$ for all $x\in U_0$.

\vspace{1ex}
\noindent(c) The family of disks $\{W^s_x:x\in U_0\}$ defines a topological foliation of $U_0$.
\end{prop}

\begin{proof}  See~\cite[Theorem~4.2 and Lemma~4.8]{AraujoM17}.
\end{proof}

The splitting $T_\Lambda M=E^s\oplus E^{cu}$ extends continuously to a 
splitting $T_{U_0} M=E^s\oplus E^{cu}$ where $E^s$ is the invariant uniformly contracting bundle in Proposition~\ref{prop:Es}.  
(In general, $E^{cu}$ is not invariant.)
Given $a>0$, we define the {\em center-unstable cone field},
\[
\cC^{cu}_x(a)=\{v= v^s+v^{cu}\in E^s_x\oplus E^{cu}_x:\|v^s\|\le a\|v^{cu}\|\}, \quad x\in U_0.
\]

\begin{prop} \label{prop:Ccu}
Let $\Lambda$ be a partially hyperbolic attracting set.
There exists $T_0>0$ such that for any $a>0$, after possibly shrinking  $U_0$,
\[
DZ_t\cdot \cC^{cu}_x(a)\subset \cC^{cu}_{Z_tx}(a) \quad\text{for all $t\ge T_0$, $x\in U_0$}.
\]
\end{prop}

\begin{proof}  See~\cite[Proposition~3.1]{AraujoM17}.
\end{proof}

\begin{prop} \label{prop:VE}
Let $\Lambda$ be a singular hyperbolic attracting set.
After possibly increasing $T_0$ and shrinking $U_0$, there exist
constants $K,\theta>0$ such that
$|\det(DZ_t| E^{cu}_x)|\geq K \, e^{\theta t}$ for all
$x\in U_0$, $t\geq 0$.
\end{prop}

\begin{proof}
Let $K_0,\theta_0>0$ be the constants from Definition~\ref{def:VE}.
Fix $a>0$ and $T_0$ as in Proposition~\ref{prop:Ccu}.
We may suppose without loss that $K_0<2$ and that $K_0 e^{\theta_0 T_0}>2$.

By continuity, we may assume that for every $x\in U_0$ there exists $y\in\Lambda$ such that
\[
|\det(DZ_t| P)|\ge {\textstyle \frac12}|\det(DZ_t| E^{cu}_y)|
\ge {\textstyle \frac12}K_0 e^{\theta_0t},
\]
for all $t\in[0,T_0]$ and every $d_{cu}$-dimensional subspace
 $P\subset \cC^{cu}_x(a)$.

Write $t=mT_0+r$ where $m\in\N$, $r\in(0,T_0]$.
Since $Z_{jT_0}x\in U_0$ for all $j\ge0$
by invariance of $U_0$, and since $DZ_{jT_0}P\subset \cC^{cu}_{Z_{jT_0}x}(a)$ 
for all $j\ge0$ by Proposition~\ref{prop:Ccu},
it follows inductively that 
\[
|\det(DZ_t|P)|\ge 
({\textstyle \frac12}K_0 e^{\theta_0r})
({\textstyle \frac12}K_0 e^{\theta_0T_0})^m
\ge ({\textstyle \frac12}K_0)^{1+t/T_0}e^{\theta_0t}
=Ke^{\theta t},
\]
where $\theta=T_0^{-1}\log(\frac12 K_0e^{\theta_0 T_0})>0$
and $K>0$.
Taking $P=E^{cu}_x$ yields the desired result.
\end{proof}

\section{Global Poincar\'e  map $f:X\to X$}
\label{sec:f}

In this section, we 
suppose that $\Lambda$ is a partially hyperbolic attracting set, and recall
how
to construct a piecewise smooth Poincar\'e map $f:X\to X$ preserving a contracting stable foliation $\cW^s(X)$.
This largely follows~\cite{APPV09} (see also~\cite[Chapter~6]{AraujoPacifico})
but with slight modifications; the details enable us to establish notation required for later sections.
Mainly for notational convenience we restrict to the case $d_{cu}=2$.

\subsection{Construction of the global cross-section $X$}

Let $y\in\Lambda$ be a regular point (not an equilibrium).
There exists an open set (flow box)  $V_y\subset U_0$ containing $y$ such that the flow on $V_y$ is diffeomorphic to a linear flow.
More precisely, let $\cD$ denote the $(\dim M-1)$-dimensional unit disk
and fix $\eps_0\in(0,1)$ small.
There is a diffeomorphism $\chi:\cD\times(-\eps_0,\eps_0)\to V_y$ 
with $\chi(0,0)=y$ such that
$\chi^{-1}\circ Z_t\circ\chi(z,s)=(z,s+t)$.
Define the cross-section $\Sigma_y=\chi(\cD\times\{0\})$.

For each $x\in\Sigma_y$, let $W^s_x(\Sigma_y)=
\bigcup_{|t|<\eps_0}Z_t(W^s_x)\cap \Sigma_y$.
This defines a topological foliation $\cW^s(\Sigma_y)$ of $\Sigma_y$.

We can identify $\Sigma_y$ diffeomorphically with $(-1,1)\times\cD^{d_s}$.
The stable boundary
 $\partial^s\Sigma_y\cong \{\pm1\}\times \cD^{d_s}$ consists of two stable leaves.
Let $\cD_{1/2}^{d_s}$ denote the open disk of radius $\frac12$ in $\R^{d_s}$.
Define the {\em subcross-section}
 $\Sigma'_y\cong (-1,1)\times \cD_{1/2}^{d_s}$,
and the corresponding subflow box $V'_y\cong\Sigma'_y\times(-\eps_0,\eps_0)$ consisting of trajectories in $V_y$ that pass through $\Sigma_y'$.

For each equilibrium $\sigma\in\Lambda$, we let $V_\sigma$ be an open neighborhood of $\sigma$ on which the flow is linearizable.
Let $\gamma^s_\sigma$ and $\gamma^u_\sigma$ denote the local stable and unstable manifolds of $\sigma$
within $V_\sigma$; trajectories starting in $V_\sigma$ remain in $V_\sigma$ for all future time if and only if they lie in $\gamma^s_\sigma$.

\begin{rmk}
  Note that $W^s_\sigma$ denotes the strong stable manifold
  of $\sigma$.  In general,
  $\dim \gamma^s_\sigma\ge\dim W^s_\sigma= d_s$.  (In the
  case of a Lorenz-like singularity,
  $\dim \gamma^s_\sigma=d_s+1$.)
\end{rmk}

Define $V_0=\bigcup_\sigma V_\sigma$.  We shrink the
neighborhoods $V_\sigma$ so that (i) they are disjoint, (ii)
$\Lambda\not\subset V_0$, and (iii)
$\gamma^u_\sigma\cap\partial V_\sigma\subset V'_y$ for some
regular point $y=y(\sigma)$.

By compactness of $\Lambda$, there exists $\ell\ge1$ and
regular points $y_1,\dots,y_\ell\in\Lambda$ such that
$\Lambda\setminus V_0 \subset \bigcup_{j=1}^\ell V'_{y_j}$.
We enlarge the set $\{y_j\}$ to include the points
$y(\sigma)$ mentioned in~(iii) above.  Adjust the positions
of the cross-sections $\Sigma_{y_j}$ if necessary so that
they are disjoint, and define the global cross-section
\[
\textstyle X=\bigcup_{j=1}^\ell \Sigma_{y_j}.
\]

In the remainder of the paper, we often modify the choices
of $U_0$ and $T_0$.  However, the choices of $V_{y_j}$,
$\Sigma_{y_j}$ and $X$ remain unchanged from now on and
correspond to our current choice of $U_0$ and $T_0$.  To
avoid confusion, all subsequent choices will be labelled
$U_1\subset U_0$ and $T_1\ge T_0$.  In particular, we
suppose from now on that
$U_1\subset V_0 \cup \bigcup_{j=1}^\ell V'_{y_j}$.

\subsection{Definition of the Poincar\'e map}

By Proposition~\ref{prop:Ws}, for any $\delta>0$ we can
choose $T_1\ge T_0$ such that
\begin{align} \label{eq:delta} \diam
  Z_t(W^s_x(\Sigma_{y_j}))<\delta, \quad\text{for all}\;
  x\in \Sigma_{y_j},\,j=1,\dots,\ell,\,t>T_1.
\end{align}

Define 
\[
\textstyle \Gamma_0=\{x\in X:Z_{T_1+1}(x)\in\bigcup_\sigma\gamma^s_\sigma\}, \qquad
X'=X\setminus\Gamma_0.
\]
If $x\in X'$, then $Z_{T_1+1}(x)$ cannot remain inside $V_0$ so there exists $t>T_1+1$ 
and $j=1,\dots,\ell$ such that
$Z_tx\in V'_{y_j}$.   Since $\eps_0<1$, there exists $t>T_1$ such that $Z_tx\in\Sigma_{y_j}'$.
Hence for $x\in X'$, we can define
\[
\textstyle f(x)=Z_{\tau(x)}(x)\qquad\text{where}\qquad
\tau(x)=\inf\{t>T_1:Z_tx\in\bigcup_{j=1}^\ell \Cl\Sigma_{y_j}'\}.
\]
In this way we obtain a piecewise $C^r$ global Poincar\'e map
$f:X'\to X$ with piecewise $C^r$ roof function
$\tau:X'\to[T_1,\infty)$.

\begin{lemma} \label{lem:log}
If $\Lambda$ contains no equilibria (so $\Gamma_0=\emptyset$), then $\tau\le T_1+2$.
In general, there exists a constant $C>0$ such that 
\[
\tau(x)\le -C\log\dist(x,\Gamma_0)\quad\text{for all $x\in X'$.}
\]
\end{lemma}

\begin{proof}
This is a standard result so we sketch the arguments.

If $Z_{T_1+1}x\in V_{y_j}'$ for some $j$, then $Z_tx\in \Sigma_{y_j}'$
for some $t\in(T_1+1-\eps_0,T_1+1+\eps_0)$ so $\tau(x)\le T_1+2$.
Otherwise, $Z_{T_1+1}x\in V_\sigma\subset V_0$ for some equilibrium $\sigma$, and we define
\[
\tau_0(x)=\sup\{t\in[0,T_1+1]:Z_tx\not\in V_\sigma\},\qquad
\tau_1(x)=\sup\{t\ge T_1+1:Z_tx\in V_\sigma\}.
\]
Note that $Z_{\tau_1(x)}(x)\in \bigcup_jV_{y_j}'$ 
so $\tau(x)\le \tau_1(x)+1\le \tau_1(x)-\tau_0(x)+T_1+2$.

By the Hartman-Grobman Theorem, the flow in $V_\sigma$ is
homeomorphic (by a time-preserving conjugacy) to the
linearized flow $\dot x=DG(\sigma) x=(A\oplus E)x$ where
$A$ has eigenvalues with negative real part and $E$ has
eigenvalues with positive real part.  After writing $E$ in
Jordan normal form, a standard and elementary argument shows
that the ``time of flight'' of trajectories in $V_\sigma$
satisfies
$\tau_1(x)-\tau_0(x)\le
-C'\log\dist(Z_{\tau_0(x)}(x),\Gamma')$
where $\Gamma'$ denotes the local stable manifold of
$\sigma$ in the linear flow.

Finally, we can suppose without loss that
$\partial V_\sigma$ is smooth so that the initial transition
$x\mapsto Z_{\tau_0(x)}(x)$ is a diffeomorphism in a
neighborhood of $\Gamma_0$.  Hence
$\dist(Z_{\tau_0(x)}(x),\Gamma')\approx
\dist(x,\Gamma_0)$ up to uniform constants.
\end{proof}

\begin{rmk} \label{rmk:log} 
It is immediate from the proof of Lemma~\ref{lem:log} that
$\tau(x)\to\infty$ as $\dist(x,\Gamma_0)\to0$.
\end{rmk}

Define the topological foliation $\cW^s(X)=\bigcup_{j=1}^\ell \cW^s(\Sigma_{y_j})$ of $X$ with leaves $W^s_x(X)$ passing through each $x\in X$.

\begin{prop} \label{prop:invf}
For $T_1$ sufficiently large,
$f(W^s_x(X))\subset W^s_{fx}(X)$ for all $x\in X'$.
\end{prop}

\begin{proof}
By definition of $V_{y_j}'$, it follows from~\eqref{eq:delta} that we can
choose $T_1$ large (and hence $\delta$ small) such that 
$W^s_{fx}(X)\subset V_{y_j}$ 
whenever $fx\in V'_{y_j}$.  
The result follows from this by definition of $\cW^s(X)$ and flow invariance
of $\cW^s$.
\end{proof}

Define $\partial^s X=\bigcup_{j=1}^\ell \partial^s\Sigma_{y_j}$ and let 
\[
\Gamma=\Gamma_0\cup\Gamma_1,\qquad \Gamma_1=\{x\in X':fx\in\partial^sX\}.
\]

\begin{prop} \label{prop:Gamma}
$\Gamma$ is a finite union of stable disks $W^s_x(X)$, $x\in X$.
\end{prop}

\begin{proof}
It is clear that $W^s_x(X)\subset \Gamma$ for all $x\in\Gamma$.
Also, if $x_0\not\in\Gamma$ than $fx_0=Z_{\tau(x_0)}(x_0)\in \Sigma'$ for some $\Sigma'\in\{\Sigma'_{y_j}\}$.  For $x$ close to $x_0$, it follows from continuity of the flow that
$fx\in\Sigma'$ (with $\tau(x)$ close to $\tau(x_0)$).
Hence $x\not\in\Gamma$ and so $\Gamma$ is closed.

It remains to rule out the possibility that a sequence of stable disks $W^s_{x_n}(X)$, $x_n\in\Gamma$, accumulates on $W^s_{x_0}(X)$ where $x_0=\lim_{n\to\infty}x_n$.
In showing this, it is useful to note that if $Z_tx\in V'_y$ then $Z_sx\in \Sigma'_y$ for some $s\in(t-1,t+1)$.  In particular, if $Z_tx\in V'_y$ for some $t\ge T_1+1$, then
$\tau(x)\le t+1$.

There are two cases to consider:

\noindent
{\bf Case 1:}
$Z_{T_2}x_0\in V'_y$ for some
$T_2\ge T_1+1$, $y\in\{y_1,\dots,y_\ell\}$.  In this case, restricting to large $n$ we have $Z_{T_2}x_n\in V'_y$, and hence $\tau(x_n)\le T_2+1$.
It follows that $\bigcup_n W^s_{x_n}(X)\subset X\cap\bigcup_j\bigcup_{t\in[0,T_2+1]}Z_{-t}(\partial^s\Sigma_{y_j})$.  But this is a compact submanifold of $X$ with the same dimension $d_s$ as the stable disks, so $\{x_n\}$ is finite.

\noindent{\bf Case 2:}
$Z_tx_0\in V'_\sigma$ for all $t\ge T_1+1$ for some equilibrium $\sigma$.
Note that $Z_tx_0\in\gamma_\sigma^s$ for all $t\ge T_1+1$.
As in Case 1, we can easily rule out accumulations when $\tau(x_n)\le T_1+1$ so we can suppose that $\tau(x_n)>T_1+1$.  Also, $\gamma^s_\sigma\cap Z_{T_1+1}(X)$ is a compact submanifold of dimension $d_s$, so $Z_{T_1+1}x_n\in V'_\sigma\setminus \gamma^s_\sigma$.
Hence the trajectory through $Z_{T_1+1}x_n$ eventually leaves $V'_\sigma$ close to $\gamma^u_\sigma$.  Such trajectories immediately enter the flow box $V'_{y(\sigma)}$ and hence hit $\Sigma'_{y(\sigma)}$.   In particular, $f(x_n)\in \Sigma'_{y(\sigma)}$ and $x_n\not\in\Gamma$.
\end{proof}

Let $X''=X\setminus\Gamma$.
Then $X''=S_1\cup\dots\cup S_m$ for some $m\ge1$, where each $S_i$ is homeomorphic to $(-1,1)\times \cD^{d_s}$.
We call these regions {\em smooth strips}.
Note that $f|_{S_i}:S_i\to X$ is a diffeomorphism onto its image and
$\tau|_{S_i}:S_i\to[T_1,\infty)$ is smooth for each $i$.
The foliation $\cW^s(X)$ restricts to a foliation $\cW^s(S_i)$ on each $S_i$.

\begin{rmk} \label{rmk:X} In future sections, it may be
  necessary to increase $T_1$ leading to changes to $f$,
  $\tau$, $\Gamma$ and $\{S_i\}$ (and the constant $C$ in
  Lemma~\ref{lem:log}).  However the global cross-section
  $X=\bigcup\Sigma_{y_j}$ continues to remain fixed
  throughout the paper.
\end{rmk}

\section{Uniform hyperbolicity of the Poincar\'e map}
\label{sec:UH}

Let $\Lambda$ be a singular hyperbolic attracting set.  We continue to assume
$d_{cu}=2$ for notational simplicity.
In this section, we show that for $T_1$ sufficiently large, the global Poincar\'e map
$f:X'\to X$ constructed in Section~\ref{sec:f} is uniformly hyperbolic
(with singularities).  (As noted in Remark~\ref{rmk:X}, the global cross-section $X=\bigcup\Sigma_{y_j}$ is independent of $T_1$.)

Let $S\in\{S_i\}$ be one of the smooth strips from the end of Section~\ref{sec:f}.
There exist cross-sections
$\Sigma$, $\wt\Sigma\in \{\Sigma_{y_j}\}$ such that
$S\subset\Sigma$ and $f(\Sigma)\subset\wt\Sigma$.

The splitting $T_{U_0}M=E^s\oplus E^{cu}$ induces a continuous
splitting $T\Sigma=E^s(\Sigma)\oplus E^u(\Sigma)$ defined by
\begin{align*}
E^s_x(\Sigma)=(E^s_x\oplus\R\{G(x)\})\cap T_x{\Sigma}
\quad\mbox{and}\quad
E^u_x(\Sigma)=E^{cu}_x\cap T_x{\Sigma},\quad x\in\Sigma.
\end{align*}
The analogous definitions apply to $\wt\Sigma$.

For each $y\in\wt\Sigma$, 
define the projection $\pi_y:T_yM=T_y\wt\Sigma\oplus\R\{G(y)\}\to T_y\wt\Sigma$.
Also, for $x\in\Sigma$, define the projection
$\hat\pi_x:E^s_x\oplus\R\{G(x)\}\to E_x^s$.  

By finiteness of the set of cross-sections $\{\Sigma_{y_j}\}$, there is a universal constant $C_1\ge1$ such that
\begin{align} \label{eq:proj} \nonumber
 & \|\pi_{y}v\|\le C_1\|v\|\quad\text{for all $v\in T_yM$},
\\
& 
\|\hat \pi_xv\|\le C_1\|v\|\quad\text{for all $v\in E^s_x\oplus\R\{G(x)\}$}.
\end{align}

\begin{prop}\label{prop:secUH}
(a) 
$Df\cdot E^s_x(\Sigma) = E^s_{fx}(\wt\Sigma)$ for all $x\in S$, and
$Df\cdot E^u_x(\Sigma) = E^u_{fx}(\wt\Sigma)$ for all $x\in\Lambda\cap S$.
\\
  (b) Let $\lambda_1\in(0,1)$.  
For $T_1$ sufficiently large 
  if $\inf \tau>T_1$, then for all $S\in\{S_i\}$,
\[
\|Df | E^s_x(\Sigma)\| \le  \lambda_1\quad\text{and}\quad
\|Df | E^u_x(\Sigma)\| \ge \lambda_1^{-1} \quad\text{for all $x\in S$}.
\]
\end{prop}

\begin{proof}
(a) For $x\in S$, we have that $Df(x):T_x\Sigma\to T_{fx}\wt\Sigma$ is given by
  \begin{align} \label{eq:Df}
    Df(x)= D(Z_{\tau(x)}(x)) =
    DZ_{\tau(x)}(x)+G(fx) D\tau(x).
  \end{align}

Let $v\in E^s_x(\Sigma)\subset E^s_x+\R\{G(x)\}$.
Then using $DZ_t$-invariance of $E^s$ on $U_0$ and of the flow direction,
\[
Df(x)v\in DZ_{\tau(x)}(x)E^s_x+DZ_{\tau(x)}(x)\R\{Gx\}+\R\{G(fx)\}
\subset E^s_{fx}+\R\{G(fx)\},
\]
so $Df(x)v\in (E^s_{fx}+\R\{G(fx)\})\cap T_{fx}\wt\Sigma=E^s_{fx}(\wt\Sigma)$.

Similarly, for $x\in \Lambda\cap S$ and $v\in E^u_x(\Sigma)\subset E^{cu}_x$, using $DZ_t$-invariance of $E^{cu}$ on $\Lambda$ and the fact that the flow direction lies in $E^{cu}$,
\[
Df(x)v\in DZ_{\tau(x)}(x)E^{cu}_x+\R\{G(fx)\} \subset E^{cu}_{fx},
\]
so $Df(x)v\in E^{cu}_{fx}\cap T_{fx}\wt\Sigma=E^u_{fx}(\wt\Sigma)$.

\vspace{1ex}
\noindent (b) 
By~\eqref{eq:Df} and the definition of $\pi_y$,
\begin{align} \label{eq-Df}
Df(x)= \pi_{fx}Df(x)=\pi_{fx}DZ_{\tau(x)}(x)\quad\text{
for $x\in S$}.
\end{align}
Using the definition of $\hat\pi_x$, for $v\in E^s_x(\Sigma)\subset E^s_x\oplus \R\{G(x)\}$,
\begin{align*}
\|Df(x)v\|
=\|\pi_{fx}DZ_{\tau(x)}(x)\hat\pi_xv\|
\le C_1^2\|DZ_{\tau(x)}(x)| E^s_x\|\,\|v\|,
\end{align*}
by~\eqref{eq:proj}.
It follows that
\[
\|Df | E^s_x(\Sigma)\| \le C_1^2C\lambda^{\tau(x)}\le C_1^2C\lambda^{T_1}.
\]
where $C>0$, $\lambda\in(0,1)$ are as in~\eqref{eq:contract}.
The first estimate in (b) is immediate for $T_1$ large enough.

For the second estimate,
define $P=DZ_{\tau(x)}E^{cu}_x$ and write $DZ_{\tau(x)}(x):E^{cu}_x\to P$ in coordinates corresponding to the splittings
\[
E^{cu}_x=E^u_x(\Sigma)\oplus \R\{G(x)\}, \qquad
P=(P\cap\wt\Sigma)\oplus \R\{G(fx)\}.
\]
In these coordinates, it follows from invariance and neutrality of the flow direction that
\[
DZ_{\tau(x)}(x)= \left(\begin{array}{cc} a_{11}(x) & 0 \\ a_{21}(x) & a_{22}(x) \end{array}\right),
\]
where $\sup_x |a_{22}(x)|\le C_2$ for some
constant $C_2>0$.
Moreover, 
by~\eqref{eq-Df},
\[
a_{11}(x)=\pi_{fx}DZ_{\tau(x)}(x)|_{E^u_x(\Sigma)}=Df(x)|_{E^u_x(\Sigma)}.
\]
Hence by Proposition~\ref{prop:VE},
\begin{align*}
|Df(x)| E^u_x(\Sigma)|  =
|a_{11}(x)|  & \ge C_2^{-1}|\det DZ_{\tau(x)}(x)| E^{cu}_x|
\\ & \ge 
C_2^{-1}K e^{\theta \tau(x)}
 \ge C_2^{-1}K e^{\theta T_1}
\ge \lambda_1^{-1},
\end{align*}
for $T_1$ sufficiently large.
\end{proof}

Next, for $a>0$ we define the {\em unstable cone field} 
\begin{align*}
\cC^u_x(\Sigma,a)=\{w=w^s+w^u\in E^s_x(\Sigma)\oplus E^u_x(\Sigma):
 \|w^s\| \le a \|w^u\| \}, \quad x\in \Sigma.
\end{align*}

\begin{prop} \label{prop:sec-cone}
For any $a>0$, $\lambda_1\in(0,1)$, we can increase $T_1$ and shrink $U_1$ such that if $\inf\tau>T_1$ then for all $S\in\{S_i\}$
\\
(a)
$Df(x)\cdot \cC^u_x(\Sigma,a) \subset \cC^u_{fx}(\Sigma,a)$ 
for all $x\in S$.
\\
(b)
  $\| Df(x)w\| \ge \lambda_1^{-1} \|w\|$
  for all $x\in S$, $w\in \cC^u_x(\Sigma,a)$.
\end{prop}

\begin{proof}
Let $w=w^s+w^u\in\cC^u_x(\Sigma,a)$.
The estimates in Proposition~\ref{prop:secUH}(b) hold with $\lambda_1=1$, so
\[
\|Df(x)w^s\|\le \|w^s\|\le a\|w^u\|\le a\|Df(x)w^u\|,
\]
proving (a).
\\
(b) 
Let $\lambda_1\in(0,1)$ be the constant in Proposition~\ref{prop:secUH}(b).
For $w\in\cC^u_x(\Sigma,a)$,
\[
\|Df(x)w\|\ge 
(1-a)\lambda_1^{-1}\|w^u\|
\ge (1-a)(1+a)^{-1}\lambda_1^{-1}\|w\|
\]
Since $\lambda_1$ is arbitrarily small,
the result follows with a new value of $\lambda_1$.
\end{proof}

Taking unions over smooth strips $S$ and cross-sections $\Sigma$, we obtain a global 
  continuous uniformly hyperbolic splitting
  \[
TX''=E^s(X)\oplus E^u(X),
\]
with the following properties:

\begin{thm} \label{thm:global}
  The stable bundle $E^s(X)$ and the restricted splitting
  $T_\Lambda X''=E^s_\Lambda(X)\oplus E^u_\Lambda(X)$ are
$Df$-invariant.  

Moreover,
for fixed $a>0$, $\lambda_1\in(0,1)$, we can arrange that
\[
Df\cdot \cC^u_x(X,a)\subset \cC^u_{fx}(X,a)\quad\text{and}\quad
\|Df(x)w\|\ge \lambda_1^{-1}\|w\|
\]
 for all $x\in X''$, $w\in\cC^u_x(X,a)$.
 \qed
\end{thm}

\section{The stable lamination is a topological foliation}
\label{sec:fol}

The stable manifold theorem guarantees the existence of an $Z_t$-invariant stable lamination
consisting of smoothly embedded disks $W^s_x$ through each point $x\in\Lambda$.   For general partially hyperbolic attracting sets, there is no guarantee that $\{W^s_x:x\in\Lambda\}$ defines a topological foliation in an open neighborhood of $\Lambda$.
However, in this section we show that this is indeed the case under our assumptions that $\Lambda$ is singular hyperbolic with $d_{cu}=2$:

\begin{thm} \label{thm:fol}
Let $\Lambda$ be a singular hyperbolic attracting set with $d_{cu}=2$.
Then the stable lamination $\{W^s_x:x\in\Lambda\}$ is a topological foliation of an open neighborhood of~$\Lambda$.
\end{thm}

The method of proof is to show that $\{W^s_x:x\in\Lambda\}$ coincides with the topological foliation
$\{W_x^s:x\in U_0\}$ in Proposition~\ref{prop:Ws}(c).
In particular, we have {\em a posteriori} that $\Lambda\subset\Int \bigcup_{x\in\Lambda}W^s_x$.  The proof shows that for every $x$ in an open neighbourhood of $\Lambda$, there exists $z\in\Lambda$ such that $x\in W^s_z$ (and hence $W^s_x=W^s_z$).

Fix $a>0$ as in Theorem~\ref{thm:global}.
A smooth curve $\gamma:[0,1]\to \Sigma\subset X$ is called a {\em $u$-curve} if
$D\gamma(t)\in \cC^u_{\gamma(t)}(\Sigma,a)$ for all $t\in[0,1]$.
  We say that a $u$-curve $\gamma$ contained in $X$
  \emph{crosses} a smooth strip $S$ if each stable leaf
  $W^s_x(S)$ intersects $\gamma$ in a unique point.

  \begin{prop}\label{prop:cross}
    For every $u$-curve $\gamma_0$ there exists $n\ge1$ and
    a restriction $\hat\gamma\subset\gamma_0$ so that
    $f^n|_{\hat\gamma}: \hat\gamma\to f^n\hat\gamma$ is a
    diffeomorphism and $f^n\hat\gamma$ crosses $S_j$
    for some $j$.
  \end{prop}

  \begin{proof}
We choose $\lambda_1\in(0,\frac14]$.
Let $S\in\{S_1,\dots,S_m\}$ and let
$\gamma$ be a $u$-curve in $S$ with length $|\gamma|$.
We consider three possibilities:
\begin{itemize}
\item[(i)] $f\gamma\subset S_i$ for some $i$.
In this case $|f\gamma|\ge 4|\gamma|$ by Theorem~\ref{thm:global}.
\item[(ii)] $f\gamma$ intersects $\bigcup \partial S_i$ in precisely one point $q$.  In this case at least one of the connected components of
$f\gamma\setminus \{q\}$ has length at least $2|\gamma|$.
\item[(iii)] $f\gamma$ intersects $\bigcup \partial S_i$ in at least two points. 
\end{itemize}
In case~(iii), we are finished with $n=1$.
In the other cases, we can pass to a restriction $\tilde\gamma$ such that
$\tilde\gamma$ and $f\tilde\gamma$ lie in smooth strips with
$|f\tilde\gamma|\ge 2|\gamma|$.

    By Theorem~\ref{thm:global},
$f\tilde\gamma$ is a $u$-curve so we can repeat the procedure.
After one such repetition,
either the process has terminated with $n=2$ or 
there is a restriction $\tilde\gamma$ such that
$\tilde\gamma$ and $f^2\tilde\gamma$ lie in smooth strips with
$|f^2\tilde\gamma|\ge 4|\gamma|$.
Since $X$ is bounded, the process terminates in finitely many steps.
  \end{proof}

\begin{prop}\label{prop:denseper}
  There exists a finite set $\{p_1,\dots,p_k\}\subset X\cap\Lambda$ 
  such that each $p_i$ is a periodic point for $f$ and
$\bigcup_{n\ge0}f^{-n}\big(\bigcup_{i=1}^k W^s_{p_i}(X)\big)$
is dense in $X$.
\end{prop}

\begin{proof}
Let $\gamma_0$ be a $u$-curve lying in a smooth strip.
By Proposition~\ref{prop:cross}, $f^{n_1}\gamma_0$ crosses a smooth strip for some $n_1\ge1$.
Moreover, there exists a restriction $\tilde\gamma_0\subset\gamma_0$ such that
$f^{n_1}$ maps $\tilde\gamma_0$ diffeomorphically inside this strip.  Applying Proposition~\ref{prop:cross}
again, we obtain $n_2>n_1$ such that $f^{n_2}\gamma_0$ crosses a strip.
Inductively, we obtain $1\le n_1<n_2<\cdots$ such that
$f^{n_j}\gamma_0$ crosses a strip for each $n_j$.
Since the number of smooth strips is finite, there exists $1\le q_1<q_2$ such that 
$f^{q_1}\gamma_0$ and $f^{q_2}\gamma_0$ cross the same smooth strip~$S$.

Let $q=q_2-q_1$, $\gamma=f^{q_1}\gamma_0$.  Choose a restriction 
$\tilde\gamma$ of $\gamma$ such that
  $f^q|_{\tilde\gamma}: \tilde\gamma \to f^q\tilde\gamma$
  is a diffeomorphism and 
  $f^q\tilde\gamma$ crosses $S$.
Shrink $\gamma$ and $\tilde\gamma$ if necessary so that 
$\gamma$ and $f^q\tilde\gamma$ cross $\Cl S$ and are contained in $\Cl S$.

Define  the surjection $g:\tilde\gamma\to \gamma$ such that $g(x)$ is the unique point where $W^s_{f^qx}(X)$ intersects $\gamma$.   
Since $W^s(X)$ restricts to a topological foliation of $S$, it follows that $g$ is continuous.  Also $\tilde\gamma\subset\gamma$ are one-dimensional curves, so by the intermediate value theorem 
$g$ possesses a fixed point $x_0\in\Cl{\tilde\gamma}$.

Since $g(x_0)=x_0$ it follows that $f^q x_0\subset W^s_{x_0}(X)$ and hence that
$f^q(W^s_{x_0}(X))\subset W^s_{x_0}(X)$.
By~\eqref{eq:delta}, $f^q:W^s_{x_0}(X)\to W^s_{x_0}(X)$ is a strict contraction, so
$f^qp=p$ for some $p\in W^s_{x_0}(X)$.  In particular, $p$ is a periodic point for $f$ lying in $X\cap U_0$.
Since $\Lambda$ is an attracting set, $p\in X\cap\Lambda$.
Moreover, $f^{q_1}\gamma_0$ intersects $W^s_p(X)$.

Starting with a new $u$-curve $\gamma_0'$ and proceeding as before, either
$f^n\gamma_0'$ crosses $S$ and hence intersects $W^s_p(X)$ for some $n\ge0$, or we can construct a new periodic orbit $p'$ in a new smooth strip such that $f^n\gamma_0'$ intersects $W^s_{p'}(X)$.   In this way we obtain periodic points $p_1,\dots,p_k$ 
such that every $u$-curve eventually intersects $\bigcup_{i=1}^k W^s_{p_i}(X)$ 
under iteration.  Since $u$-curves are dense and arbitrarily short, the result follows.
  \end{proof}

  \begin{rmk} The periodic points constructed in the proof
    of Proposition~\ref{prop:denseper} lie in distinct
    smooth strips, so $k\le m$.  The proof does not show
    that each strip contains a periodic point.
\end{rmk}

\begin{prop}\label{prop:fol}
  For each $x\in X$ there exists $y\in X\cap\Lambda$ such that $x\in W^s_y(X)$.
\end{prop}

\begin{proof}
Define 
\[
\textstyle E=\{ x\in X : x\in W^s_y(X)\;\text{for some $y\in X\cap\Lambda$}\}.
\]
 We show that $E=X$.

  Suppose first that $x\in \bigcup_{n\ge0}f^{-n}\big(\bigcup_{i=1}^k
  W^s_{p_i}(X)\big)$, so there exists $n\ge0$, $i\in\{1,\dots,k\}$
and $y\in W^s_{p_i}(X)$ such that
  $f^nx=y$.
Choose an open set $V$ formed of a union of stable leaves and containing $x$ such that $f^n|_V:V\to f^nV$ is a diffeomorphism.
By Remark~\ref{rmk:per},
periodic points are not isolated inside $X\cap \Lambda$, so there exists a
 sequence  $W_j$ of stable leaves inside $f^nV\cap E$ that converges to $W^s_{p_i}(X)$.
Choose $y_j\in W_j$ such that $y_j\to y$.  Let $x_j=f^{-n}y_j$ so 
$x_j\to x$.

Since $y_j\in E$, we have
  $y_j\in W^s_{y'_j}(X)$ for some $y'_j\in f^nV\cap\Lambda$.
Write $y'_j=f^nx_j'$ where $x_j'\in V\cap\Lambda$.
Since $f^n|_V$ is a diffeomorphism and 
$y_j\in W^s_{y'_j}(X)$, it follows that $x_j\in W^s_{x'_j}(X)$.

Passing to a subsequence if needed, we can assume that
$x'_j\to x'\in X\cap\Lambda$ and so $x_j\to x
\in W^s_{x'}(X)\subset E$.

We have shown that $E$ contains 
$\bigcup_{n\ge0}f^{-n}\big(\bigcup_{i=1}^k W^s_{p_i}(X)\big)$
and so $E$ is dense in $X$ by Proposition~\ref{prop:denseper}.

Now for $x\in\Sigma$ we take $x_k\in E$ so that $x_k\to x$.
We know that $x_k=W^s_{y_k}(\Sigma)$ for $y_k\in A$ and
passing to a subsequence we find $y\in A\cap\Sigma$ so that
$y_k\to y$. Then $x\in W^s_{y}(\Sigma)$ and $x\in E$.
\end{proof}

\begin{pfof}{Theorem~\ref{thm:fol}}
If $x\in W^u_\sigma$ for some equilibrium $\sigma$, then $x\in\Lambda$ and there is nothing to do.  Otherwise, restricting to a smaller positively invariant neighborhood $U_0$, we can ensure that there always exists $t>0$ such that $Z_{-t}x$ lies in one of the flow boxes $V_{y_j}$.  But then there exists $t>0$ such that $Z_{-t}x\in \Sigma_{y_j}\subset X$.  By Proposition~\ref{prop:fol}, $Z_{-t}x\in W^s_y(X)$ for some $y\in X\cap\Lambda$.
Hence there exists $t>0$ such that $Z_{-t}x\in W^s_y$, and so
$x\in W^s_z$ where $z=Z_ty\in\Lambda$.
\end{pfof}

We have shown that the stable lamination
$\{W^s_x:x\in\Lambda\}$ coincides with the stable foliation 
$\{W^s_x:x\in U_0\}$.
From now on, we refer to $\cW^s=\{W^s_x:x\in\Lambda\}$ as the stable foliation.

\section{H\"older regularity and absolute continuity of the stable foliation}
\label{sec:regWs}

In this section, we continue to assume that $\Lambda$ is a singular hyperbolic attracting set, and show that the topological foliation $\cW^s$ is in fact a H\"older foliation (bi-H\"older charts).  Also we recall results on absolute continuity of the stable foliation.
These results do not use explicitly the fact that $d_{cu}=2$; it suffices that the conclusion of Theorem~\ref{thm:fol} holds.

A key ingredient is regularity of stable holonomies.
Let $Y_0,\,Y_1\subset U_0$ be two smooth disjoint $d_{cu}$-dimensional
disks that are transverse to the stable foliation $\cW^s$.  Suppose that 
for all $x\in Y_0$, the stable leaf $W^s_x$ intersects each of $Y_0$ and $Y_1$ in precisely one point.
The {\em stable holonomy} $H:Y_0\to Y_1$ is given by defining $H(x)$ to be the intersection point of $W^s_x$ with $Y_1$.

\begin{lemma} \label{lem:PSW} 
There exists $\eps>0$ such that the stable holonomies $H:Y_0\to Y_1$ are $C^\eps$.   Moreover, if the angles between $Y_i$ and stable leaves are bounded away from zero for $i=0,1$, then there is a constant $K>0$ dependent on this bound but otherwise independent of the holonomy $H:Y_0\to Y_1$ such that
$d(H(y),H(y'))\le Kd(y,y')^\eps$ for all $y,y'\in Y_0$.
\end{lemma}

\begin{proof}
By Theorem~\ref{thm:fol}, we can view $\cW^s$ as the stable lamination
corresponding to the invariant splitting
$T_\Lambda M=E^s\oplus E^{cu}$ for the 
partially hyperbolic diffeomorphism $f=Z_1$.   
Hence we can apply~\cite[Theorem~A']{PSW97}.
The result in~\cite{PSW97} is formulated slightly differently in terms of a splitting $T_\Lambda M=E^s\oplus E^c\oplus E^u$,  but their 
proof covers our situation (with the invariant splitting $T_\Lambda M=E^u\oplus E^{cs}$ there replaced by the symmetric situation $T_\Lambda M=E^s\oplus E^{cu}$).
\end{proof}

\begin{thm} \label{thm:holder} 
The stable foliation $\cW^s$ is $C^\eps$ for some $\eps>0$.
\end{thm}

\begin{proof}
   Let $\{\gamma(x):x\in U_0\}$ be the family of embeddings $\gamma(x):\cD^{d_s}\to W^s_x$ described in Proposition~\ref{prop:Ws}.

Let $x\in U_0$ and
choose an embedded $d_{cu}$-dimensional disk $Y_0\subset M$
containing $x$ and transverse to $W^s_x$.  By continuity of
$E^s$, we can shrink $Y_0$ so that $Y_0$ is transverse to
$W^s_y$ at $y$ for all $y\in Y_0$.  Let $\psi:\cD^{d_{cu}}\to Y_0$ be
a smooth embedding.
The proof of~\cite[Lemma~4.9]{AraujoM17} shows that
the map $\chi:\cD^{d_s}\times\cD^{d_{cu}}\to U_0$ given by
\[
\chi(u,v)=\gamma(\psi(v))(u)
\]
is a topological chart for $\cW^s$ at $x$. Note that $\chi$ maps
horizontal lines $\{v={\rm const.}\}$ homeomorphically onto
stable disks $W^s_{\psi(v)}$.

Moreover, we claim that $\chi$ maps vertical lines $\{u={\rm const.}\}$ onto smooth transversals $Y_u$ to $\cW^s$.
To see this, we recall the notation $\gamma(y)(u)=Q(u,\varphi_y(u))$
from the proof of~\cite[Lemma~4.8]{AraujoM17}.  Here
$Q=Q_{x,0}:\R^d\to M$ is a diffeomorphism and $Q^{-1}(W^s_y)$ is given by the graph of $\varphi_y:\cD^{d_s}\to\cD^{d_{cu}}$.
Hence
\begin{align*}
Y_u  =\{\chi(u,v):v\in\cD^{d_{cu}}\}
& =\{\gamma(\psi(v))(u):v\in\cD^{d_{cu}}\}
\\ & =\{\gamma(y)(u):y\in Y_0\}
=Q\{(u,\varphi_y(u)):y\in Y_0\}.
\end{align*}
The curves $W^s_y$ foliate $U_0$, so the curves $Q^{-1}(W^s_y)=\{(u,\varphi_y(u))\}$ foliate $\cD^{d_s}\times\cD^{d_{cu}}$.   Hence 
the set $\{(u,\varphi_y(u)):y\in Y_0\}$ is precisely $\{u={\rm const.}\}$ and so
$Y_u=Q(\{u={\rm const.}\})$ verifying the claim.

Moreover, via the diffeomorphism $Q$, the angles of $Y_u$ with stable disks $W^s_y$ are bounded away from zero.   Hence for any $u\neq0$, the stable holonomy 
$H_u:Y_0\to Y_u$ satisfies
$d(H_u(y),H_u(y'))\le Kd(y,y')^\eps$ for all $y,y'\in Y_0$ by Lemma~\ref{lem:PSW}.
Also $H_u^{-1}:Y_u\to Y_0$ is a stable holonomy, so
$d(H_u^{-1}(y),H_u^{-1}(y'))\le Kd(y,y')^\eps$ for all $y,y'\in Y_u$.

Now $\chi(u,v)=\gamma(\psi(v))(u)=H_u(\psi(v))$, so
\begin{align*}
d(\chi(u,v),\chi(u,v'))  & =
d(H_u(\psi(v)),H_u(\psi(v')))
\\ & \le K d(\psi(v),\psi(v'))^\eps\le K(\Lip\psi)^\eps \|v-v'\|^\eps.
\end{align*}
Also there is a constant $L_1>0$ such that
\begin{align*}
d(\chi(u,v'),\chi(u',v')) & =
d(\gamma(\psi(v'))(u),\gamma(\psi(v'))(u'))
\\ & \le \Lip\gamma(\psi)(v')\|u-u'\|
\le L\|u-u'\|\le L_1\|u-u'\|^\eps.
\end{align*}
Altogether, letting $M=\max\{K(\Lip\psi)^\eps,L_1\}$,
  \[
d(\chi(u,v),\chi(u',v'))\le M
  \big(\|u-u'\|^\eps+
    \|v-v'\|^\eps\big)
  \\
  \le CM 
(\|u-u'\|^2+\|v-v'\|^2)^{\eps/2}
\]
where $C>0$ is an upper bound for the homogeneous function
$\displaystyle\frac{|x|^\eps+|y|^\eps}{(|x|^2+|y|^2)^{\eps/2}}$
over the set of $(x,y)\in\R^2$ such
that $|x|^2+|y|^2=1$.
Hence $\chi$ is $C^\eps$.

Next,
\begin{align*}
\|u-u'\| & \le 
\|(u,\varphi_{\psi(v)}(u))-(u',\varphi_{\psi(v')}(u'))\|
  \\ & \le\Lip(Q^{-1})
  d(Q(u,\varphi_{\psi(v)}(u)),Q(u',\varphi_{\psi(v')}(u')))
 \\ & =\Lip(Q^{-1})d(\gamma(\psi(v))(u),\gamma(\psi(v'))(u'))=\Lip(Q^{-1})d(\chi(u,v),\chi(u',v')),
\end{align*}
and
\begin{align*}
\|v-v'\|  \le \Lip(\psi^{-1})d(\psi(v),\psi(v'))
 & = \Lip(\psi^{-1})d(H_{u'}^{-1}\chi(u',v),H_{u'}^{-1}\chi(u',v'))
\\ & \le K \Lip(\psi^{-1})d(\chi(u',v),\chi(u',v'))^\eps.
\end{align*}
Moreover,
\begin{align*}
\|(u',\varphi_{\psi(v)}(u'))-(u',\varphi_{\psi(v')}(u'))\|
& =\|\varphi_{\psi(v)}(u')-\varphi_{\psi(v')}(u')\|
\\ & \le
\|\varphi_{\psi(v)}(u')-\varphi_{\psi(v)}(u)\|
+\|\varphi_{\psi(v)}(u)-\varphi_{\psi(v')}(u')\|
\\ & \le
L\|u-u'\| +\|\varphi_{\psi(v)}(u)-\varphi_{\psi(v')}(u')\|
\\ & \le
(L+1)\|(u,\varphi_{\psi(v)}(u))-(u',\varphi_{\psi(v')}(u'))\|,
\end{align*}
so
\[
d(\chi(u',v),\chi(u',v'))\le (L+1)\Lip Q\Lip(Q^{-1}) d(\chi(u,v),\chi(u',v')).
\]

Combining these estimates, we obtain
$\|(u,v)-(u',v')\|\le{\rm const.}\,d(\chi(u,v),\chi(u',v'))^\eps$.
Hence
$\|\chi^{-1}(p)-\chi^{-1}(p')\|\le{\rm const.}\,d(p,p')^\eps$ for 
$p,p'\in U_0$, so $\chi^{-1}\in C^\eps$.
\end{proof}

\begin{thm}\label{thm:H}
  The stable holonomy $H:Y_0\to Y_1$ is absolutely continuous. That is,
  $m_1\ll H_*m_0$ where $m_i$ is Lebesgue measure
  on $Y_i$, $i=0,1$.

Moreover, the Jacobian  $JH:Y_0\to\R$ given by
  \begin{align*}
JH(x)=\frac{dm_1}{dH_*m_0}(Hx)=\lim_{r\to0}\frac{m_1(H(B(x,r)))}{m_0(B(x,r))},\quad x\in Y_0,
  \end{align*}
  is bounded above and below and is $C^\eps$ for some $\eps>0$.
\end{thm}

\begin{proof}
  This essentially follows from \cite[Theorems 8.6.1 and
  8.6.13]{BarreiraPesin}.  See also~\cite[Theorem~2.1]{PughShub72} and~\cite[Section~III.3]{Mane}.
The results there are formulated under a condition of the type 
$\sup_{x\in\Lambda}\|DZ_t | E^s_x\| \sup_{x\in\Lambda} \|DZ_{-t} | E^{cu}_{Z_tx}\| \le C \lambda^t$ which is more restrictive than the domination condition~\eqref{eq:domination}.
However, it is standard that such results generalise to our setting.
(See the remark in~\cite{PughShub72} after their theorem.  Most of the required result is covered by~\cite{PughShub72} except that H\"older continuity of $JH$ is not mentioned, only continuity.)
\end{proof}


\section{One-dimensional quotient map $\bar f:\bX\to\bX$}
\label{sec:barf}

In this section, we continue to suppose that $\Lambda$ is a singular hyperbolic attracting set with $d_{cu}=2$.
Let $f:X'\to X$ be the global Poincar\'e map defined in Section~\ref{sec:f}
with invariant stable foliation $\cW^s(X)$.
We now show how to obtain a one-dimensional
piecewise $C^{1+\eps}$ uniformly expanding quotient map
$\bar f:\bX'\to\bX$.

We begin by analysing the stable holonomies for $f$.
Let $\gamma_0,\,\gamma_1\subset X$ be two $u$-curves such that 
for all $x\in \gamma_0$, the stable leaf $W^s_x(X)$ intersects each of $\gamma_0$ and $\gamma_1$ in precisely one point.
The {\em (cross-sectional) stable holonomy} $h:\gamma_0\to \gamma_1$ is given by defining $h(x)$ to be the intersection point of $W^s_x(X)$ with $\gamma_1$.

  \begin{lemma}\label{lem:h}
    The stable holonomy $h$ is $C^{1+\eps}$ for some $\eps>0$.
  \end{lemma}

\begin{proof}
Recall that $X=\bigcup \Sigma_{y_j}$ where $\Sigma_{y_j}$ is the cross-section
associated to the flow box $V_{y_j}$ for each $j$.  Since the result is local,
we can suppose that $\gamma_0,\gamma_1\subset \Sigma_{y_j}$ for some $j$ and we can choose coordinates so that the local flow $Z_t$ is linear.

Consider the $2$-dimensional disks
$Y_i=\bigcup_{t\in[-\delta_i,\delta_i]}Z_t(\gamma_i)
=\gamma_i\times[-\delta_i,\delta_i], i=0,1$, for fixed
$\delta_i>0$.  These are smooth transversals to the stable
foliation $\cW^s$ of the flow.  Provided $\delta_0$ is small
with respect to $\delta_1$, we can then consider the
holonomy $H:Y_0\to Y_1$ as in Section~\ref{sec:regWs}.

For $p=(v,0)\in \gamma_0\subset Y_0$ we
write $H(v,0)=(H_1(v),\xi(v))$ with $H_1:\gamma_0\to \gamma_1$
and $\xi:\gamma_0\to[-\delta_1,\delta_1]$. Clearly $h=H_1$
by construction. Since $\cW^s$ is flow invariant, 
\begin{align*}
  H(v,t)=(h(v),\xi(v)+t).
\end{align*}
Let $\lambda_i=(\pi_i)_*m_i$ denote Lebesgue measure on
$\gamma_i, i=0,1$, where $\pi_i:Y_i\to \gamma_i$ is
the natural projection.  
By Theorem~\ref{thm:H}, $m_1\ll H_*m_0$.  Since $\pi_1H=h\pi_0$,
\begin{align*}
\lambda_1=(\pi_1)_*m_1\ll (\pi_1H)_*m_0=(h\pi_0)_*m_0= h_*\lambda_0.
\end{align*}
Hence $h$ is absolutely continuous.

Taking balls $B(x,r)$ to be rectangles, we have for $r$ sufficiently small
\[
H(B(x,r))=\bigcup_{v'\in B(v,r)}\{h(v')\}\times(t+\xi(v')-r,t+\xi(v')+r).
\]
By Fubini,
\[
\frac{m_1(H(B(x,r)))}{m_0(B(x,r))}
=\frac{2r\lambda_1(h(B(v)))}{2r\lambda_0(B(v,r))}
=\frac{\lambda_1(h(B(v,r)))}{\lambda_0(B(v,r))},
\]
showing that $JH(x)=Jh(v)$ for all $x=(v,t)$.
By Theorem~\ref{thm:H}, $Jh$ is H\"older.   But $\dim\gamma_0=\dim\gamma_1=1$, so $Jh=|Dh|$ and the result follows.
\end{proof}

Recall that $X$ is a union of finitely many cross-sections
$\Sigma_{y_j}$, and that $f$ is smooth on a subset $X''\subset X'\subset X$ which is obtained from $X$ by removing finitely many stable leaves.
Moreover, each $\Sigma_{y_j}\cap X''$ is a union of finitely many connected smooth strips $S$ such that
$f|_S:S\to f(S)$ is a diffeomorphism.

For each $j$, let $\gamma_j\subset \Sigma_{y_j}$ be a $u$-curve crossing $\Sigma_{y_j}$.
Define $\bX=\bigcup_j\Cl\gamma_j$ 
and $\bX'=X'\cap\bX$.   
Given a smooth strip $S\subset \Sigma_{y_j}$,
there exists $k$ such that
$f(S)\subset\Sigma_{y_k}$.  
Also $f(\gamma_j)$ is a $u$-curve by Theorem~\ref{thm:global}.
Let 
$h:f(S\cap \gamma_j)\to \gamma_k$ 
be the associated stable holonomy and
define $\bar f(x)=h(fx)$ for $x\in S\cap\gamma_j$.  In this way we obtain
a one-dimensional map $\bar f:\bX'\to \bX$.

\begin{cor}   \label{cor:barf}
The quotient map $\bar f:\bX'\to\bX$ is piecewise $C^{1+\eps}$ and consists of finitely many monotone $C^{1+\eps}$ branches.  Choosing $T_1$ in Section~\ref{sec:UH} sufficiently large, we have $|D\bar f|\ge2$ on $\bX'$.
\end{cor}

\begin{proof} Since $f$ is smooth on smooth strips and the holonomies 
$h:f(S\cap \gamma_j)\to \gamma_k$ are 
$C^{1+\eps}$ by Lemma~\ref{lem:h}, it follows that $\bar f$ is piecewise $C^{1+\eps}$.
The collection of intervals $S\cap\gamma_j$ is finite, so $\bar f$ has finitely many branches.  By finiteness of the collection $\{\Sigma_{y_j}\}$, there is a constant $c>0$ such that all the holonomies $h$ considered above satisfy $|Dh|\ge c$.
Hence taking $\lambda_1$ sufficiently small in Theorem~\ref{thm:global}, we can ensure that $|D\bar f|$ is as large as desired.
\end{proof}

\section{Statistical properties for $\bar f$ and $f$}
\label{sec:stat}

In this section, we investigate statistical properties for the 
$(d_s+1)$-dimensional Poincar\'e map
$f:X'\to X$ and the one-dimensional quotient map
$\bar f:\bX'\to\bX$.
Define $\pi:X\to\bX$ H\"older by letting $\pi(x)$ be the point where $W^s_x(X)$ intersects $\bX$.  Then $\pi$ defines a semiconjugacy between $f$ and $\bar f$.

From now on, we write $f:X\to X$ and $\bar f:\bX\to\bX$ with the understanding that $f$ and $\bar f$ are not defined everywhere (and are piecewise smooth where defined).

\subsection{Spectral decomposition and physical measures}
\label{sec:spectral}

\begin{prop} \label{prop:spectral}
There exists a finite number of 
ergodic absolutely continuous $\bar f$-invariant probability measures $\bar\mu_1,\dots,\bar\mu_s$ whose basins cover a subset of $\bX$ of full Lebesgue measure.   For each $j$, the density $d\bar\mu_j/d\Leb$ lies in $L^\infty$
and $\Int\supp\bar\mu_j\neq\emptyset$.
\end{prop}

\begin{proof}
By Corollary~\ref{cor:barf}, $\bar f$ is a piecewise $C^{1+\eps}$ uniformly expanding one-dimensional map.
Hence, most of the result is immediate from~\cite[Theorem~3.3]{Keller85}.
We refer to~\cite[Lemma~3.1]{Saussol00}
for the fact that $\Int\supp\bar\mu_j\neq\emptyset$.
\end{proof}

\begin{cor} \label{cor:spectral}  There exists a 
finite number of 
ergodic $f$-invariant probability measures $\mu_1,\dots,\mu_s$ whose basins cover a subset of $X$ of full Lebesgue measure.   
Moreover, $\pi_*\mu_j=\bar\mu_j$ for each $j$.
\end{cor}

\begin{proof}
This follows from the existence of the stable foliation $\cW^s(X)$ (here, the fact that it is a topological foliation suffices) combined with Proposition~\ref{prop:spectral}.
For details, see~\cite[Sections~6.1 and~6.2]{APPV09}.
\end{proof}

\subsection{Existence of an inducing scheme}
\label{sec:induce}

In this subsection, we suppose without loss that there is a unique absolutely continuous $\bar f$-invariant measure $\bar\mu$ in Proposition~\ref{prop:spectral}
(so $s=1$).

\begin{prop} \label{prop:barf}
There exists $k\ge1$ such that 
$\supp\bar\mu=\bX_1\cup\cdots\cup\bX_k$ 
where the sets $\bX_j$ are permuted cyclically by $\bar f$, and 
$\bar f^k:\bX_j\to\bX_j$ is mixing for each $j$.

Moreover, for any $\eta\in(0,1)$, there exist constants $c,\,C>0$ such that
\[
\Big|\int_{\bX_j} v\,w\circ \bar f^{kn}\,d\bar\mu- \int_{\bX_j} v\,d\bar\mu
\int_{\bX_j} w\,d\bar\mu\Big|\le C\|v\|_{C^\eta}|w|_1 e^{-cn}\quad\text{for all $n\ge1$, $j=1,\dots,k$},
\]
for all  $v\in C^\eta(\bX)$ and $w\in L^1(\bX)$.
\end{prop}

\begin{proof} This is immediate from the quasicompactness of
  the transfer operator for $\bar f$ which is established
  in~\cite[Theorem~3.3]{Keller85}.  Indeed the result
  in~\cite{Keller85} is proved for the class of
  functions with finite $\eta$-variation (for all $\eta>0$
  sufficiently small).  This includes observables that are
  $C^\eta$.
\end{proof}

For ease of exposition, we suppose for the remainder of this
subsection that $k=1$ and $\bX_1=\bX$.
Recall that a one-dimensional map $\bF:\bY\to\bY$ is a full branch Gibbs-Markov map if
there is an at most
  countable partition $\alpha$ of $\bY$ and constants
  $C>0$, $\eps\in(0,1]$ such that for all $a\in\alpha$,
\begin{itemize}
\item $\bF|_a:a\to\bY$ is a measurable bijection, and
\item $\big|\log |D\bF(y_1)|-\log |D\bF(y_2)|\big|\le C|\bF y_1-\bF y_2|^\eps$ for
  all $y_1,y_2\in a$.
\end{itemize}

\begin{lemma} \label{lem:AFLV}
For all $\beta>0$, there exists a positive
  measure subset $\bY\subset \bX$ and 
a full branch Gibbs-Markov induced map $\bF=\bar f^\rho:\bY\to \bY$, where $\rho:\bY\to\Z^+$ is constant on partition elements and satisfies $\Leb(y\in \bY:\rho(y)>n)=O(n^{-\beta})$.
\end{lemma}

\begin{proof}
By Proposition~\ref{prop:spectral}, $\bar\mu$ is an ergodic absolutely continuous invariant probability measure on $\bX$ with $d\bar\mu/d\Leb\in L^\infty$.
  The result follows from~Theorem~\ref{thm:AFLV} provided we
  verify that $\bar\mu$ is expanding and that conditions (C0)--(C3) hold.  Let $\cS$ denote the finite
  set consisting of singularities/discontinuities of
  $\bar f$.  (In general $X\setminus\cS$ is a proper subset of $X'$ since $\cS$ includes the discontinuities of the piecewise smooth map $f$.)
Conditions~(C0) and (C3) are redundant since
  $\bar f$ is one-dimensional.  Conditions (C1) and (C2) become 
\begin{itemize}
\item[(C1)] $C^{-1}d(x,\cS)^q\le |D\bar f(x)|\le Cd(x,\cS)^{-q}$
for all $x\in\bX\setminus\cS$,
\item[(C2)]
  $\big|\log|D\bar f(x)|-\log|D\bar f(x')|\big|\le
C|x-x'|^\eta(|D\bar f(x)|^{-q}+|D\bar f(x)|^{q})$
for all $x,x'\in\bX\setminus\cS$ with $|x-x'|<\dist(x,\cS)/2$,
\end{itemize}
where $\eta\in(0,1)$ and $C,q>0$ are constants.
Since $d\bar\mu/d\Leb\in L^\infty$ 
it is immediate from (C1) that
$\log|(D\bar f)^{-1}|$ is integrable with respect to $\bar\mu$.
Also $\int\log|(D\bar f)^{-1}|\,d\bar\mu\le \log\frac12<0$ by Corollary~\ref{cor:barf},
so $\bar\mu$ is an expanding measure.

It remains to verify conditions (C1) and (C2).
Note that they are trivially satisfied for functions $\bar f$ with $D\bar f$ H\"older and bounded below.  Hence they are satisfied away from $\cS$ and also near all
discontinuity points in $\cS$.

By Proposition~\ref{prop:like}, it remains to consider
singularities $x_0\in X$ corresponding to Lorenz-like equilibria $\sigma$.
The Poincar\'e map $f$ can be written near $x_0$ as $f=h_1\circ g \circ h_2$ where $g$ corresponds to the flow near $\sigma$ and $h_1,\,h_2$ are the remaining parts of the Poincar\'e map.  In particular $D\bar h_j$ is H\"older and bounded below for $j=1,2$.

Suppose first that the flow is $C^{1+\eps}$-linearizable for some $\eps>0$
in a neighborhood of~$\sigma$.  
Incorporating the linearization into $h_1$ and $h_2$,
we can suppose without loss that the flow is linear in a neighborhood of $\sigma$. Hence
the flow is given by $x\mapsto e^{tA}x$
where $A=\lambda^u\oplus\lambda^s\oplus B$ 
with $-\lambda^u<\lambda^s<0<\lambda^u$
and $B=DG(\sigma)|E^s_\sigma$.
A standard calculation shows that in suitable coordinates,
\[
g(x,z)=(|x|^{-\lambda^s/\lambda^u},ze^{-\lambda_u^{-1}B\log|x|}).
\]
In particular, $\bar g(x)=|x|^\omega$ where
$\omega=-\lambda^s/\lambda^u\in(0,1)$.

Since $D\bar h_j$ is bounded above and below,
it follows from the chain rule that 
\[
|D\bar f(x)|\cong |D\bar g(\bar h_2x)|=\omega |\bar h_2x|^{\omega-1}\cong |x-x_0|^{\omega-1},
\]
so (C1) is satisfied.
Next,
\begin{align*}
\big|\log|D\bar f(x)|- & \log|  D\bar f(x')|\big|
 \le C_1\big( 
|x-x'|^\eps+ \big|\log|D\bar g(\bar h_2x)|-\log|D\bar g(\bar h_2x')|\big|
\\ & \qquad\qquad \qquad +|\bar g(\bar h_2x)-\bar g(\bar h_2x')|^\eps\big)
\\ & =
C_1\big(|x-x'|^\eps  
+(1-\omega)\big|\log|\bar h_2x|-\log|\bar h_2x'|\big|
+
\big||\bar h_2x|^\omega-|\bar h_2x'|^\omega\big|^{\eps}\big) .
\end{align*}
Now 
\begin{align*}
\big|\log|\bar h_2x|-\log|\bar h_2x'|\big|\le  C_2 (|\bar h_2x-\bar h_2x'|/|\bar h_2x|)^{1-\omega}
& \le C_2' |x-x'|^{1-\omega}|x-x_0|^{\omega-1}
\\ & \le C_2'' |x-x'|^{1-\omega}|D\bar f(x)|.
\end{align*}
Also without loss $|x-x_0|\le |x'-x_0|$, so
\begin{align*}
\big||\bar h_2x|^\omega-|\bar h_2x'|^\omega\big|
&\le |\bar h_2x-\bar h_2x'|(|\bar h_2x|^{\omega-1}+|\bar h_2x'|^{\omega-1})
\\ & \le C_3 |x-x'||x-x_0|^{\omega-1} \le C_3' |x-x'||D\bar f(x)|.
\end{align*}
Hence, there exists $\eta\in(0,1)$ such that
\begin{align*}
\big|\log|D\bar f(x)|-\log|  D\bar f(x')|\big|
 & \le C_4 |x-x'|^\eta(|D\bar f(x)|+1)
 \\ & \le C_4 |x-x'|^\eta(2|D\bar f(x)|+|D\bar f(x)|^{-1}),
\end{align*}
verifying (C2).

To complete the proof, we remove the assumption that the flow near $\sigma$ 
is $C^{1+\eps}$-linearizable.  By the center manifold theorem (eg.~\cite[Theorem~5.1]{HPS77}), locally we can choose a flow-invariant $C^{1+\eps}$
two-dimensional manifold $W$ tangent to $E^{cu}_\sigma$ (for some $\eps>0$).
Note that the quotient of $g|W$ coincides with $\bar g$.
By a result of
Newhouse~\cite{Newhouse17} (stated previously but without proof in~\cite{Hartman60}),
the flow restricted to $W$ (being two-dimensional) can be
 $C^{1+\eps'}$ linearized for some $\eps'>0$.  
The proof now proceeds as before.
\end{proof}

\begin{rmk}  \label{rmk:exp}
Since we have exponential decay of correlations in Proposition~\ref{prop:barf}, there is the hope of obtaining an induced Gibbs-Markov map as in
Lemma~\ref{lem:AFLV} but with exponential tails for $\rho$.  (We note that
Theorem~\ref{thm:AFLV}(2) which would give stretched exponential tails does not apply because the density $d\bar\mu/d\Leb$ is not bounded below.)
In certain situations, it is possible to construct an inducing scheme
with exponential tails by using different methods, controlling
the tail of hyperbolic times and relating this with the tail of
inducing times more directly~\cite{Gouezel06,Araujo07,AST}.
One repercussion of the existence of such an inducing scheme would be that
the error rate in the vector-valued  ASIP would be improved to 
$n^{\frac14+\eps}$ for $\eps>0$ arbitrarily small~\cite{Gouezel10}.

However, our construction here with superpolynomial tails holds in complete generality and suffices for our 
results on singular hyperbolic flows in Section~\ref{sec:SH}, so we do not pursue this further.  
\end{rmk}

\begin{prop} \label{prop:tau}
There is a constant $C>0$ such that
\[
\textstyle{\sum_{\ell=0}^{\rho(y)-1}}|\tau(f^\ell y)-\tau(f^\ell y')|\le C|\bF y-\bF y'|^\eps
\quad \text{for all $y,y'\in a$, $a\in\alpha$.}
\]
\end{prop}

\begin{proof}
It follows from the proof of Lemma~\ref{lem:AFLV} that the roof function
$\tau:X\to\R^+$ satisfies $\tau(x)=-\lambda_u^{-1}\log|\bar h_2x|+t(x)$
where $\bar h_2$ and $t$ are $C^{1+\eps}$.
Hence
\[
|\tau(x)-\tau(x')|\le \lambda_u^{-1} |\bar h_2 x-\bar h_2 x'|/|\bar h_2x|+|Dt|_\infty|x-x'|
\le C_1 |x-x'|d(x,\cS)^{-1},
\]
and the result follows from~\eqref{eq:hyptime}.
\end{proof}

\subsection{Statistical limit laws for the Poincar\'e map}

By Corollary~\ref{cor:spectral}, there is a unique ergodic
$f$-invariant probability measure $\mu$ on $X$ corresponding
to $\bar\mu$, with $\pi_*\mu=\bar\mu$.

\begin{thm} \label{thm:statf}   
Fix $\eta\in(0,1)$ and let $v\in C^\eta(X)$ with $\int_X v\,d\mu=0$.
Write $v_n=\sum_{j=0}^{n-1}v\circ f^j$.
Then the limit
$\sigma^2=\lim_{n\to\infty}n^{-1}\int_\Lambda v_n^2\,d\mu$ exists. Suppose that $\sigma^2>0$.  Then the following limit laws hold.

\begin{description}[style=unboxed,leftmargin=0cm]
\item[ASIP~\cite{CunyMerlevede15}]
Let $\eps>0$. There exists a probability space $\Omega$ supporting a sequence of random variables $\{S_n,\,n\ge1\}$ with the same joint distributions
as $\{v_n,\,n\ge1\}$,
and a sequence $\{Z_n,\,n\ge1\}$ of i.i.d.\ random variables with distribution $N(0,\sigma^2)$, such that 
\[
\textstyle \sup_{1\le k\le n} \big|S_k-\sum_{j=1}^k Z_j\big|=
O(n^\eps)
\;\text{a.e.\ as $n\to\infty$.}
\]
\item[Berry-Esseen~\cite{Gouezel05}]
There exists $C>0$ such that
\[
\big|\mu\{x\in X: n^{-1/2}v_n(x)\le a\}-\P\{N(0,\sigma^2)\le a\}\big|\le Cn^{-1/2}
 \quad\text{for all $a\in\R$, $n\ge1$.}
\]
\item[local limit theorem~\cite{Gouezel05}]
Suppose that $v$ is aperiodic (so it is not possible to write $v=c+g-g\circ f+\lambda q$ where $c\in\R$, $\lambda>0$, $g:X\to\R$ measurable and $q:X\to\Z$). 
Then for any bounded interval $J\subset\R$,
\[
\lim_{n\to\infty} n^{1/2}\mu(x\in X: v_n(x)\in J)=(2\pi\sigma^2)^{-1/2}|J|.
\]
\end{description}

For $C^\eta$ vector-valued observables $v:X\to\R^d$ with
$\int_X v\,d\mu=0$, the limit
$\Sigma=\lim_{n\to\infty}n^{-1}\int_\Lambda v_n v_n^T\,d\mu\in\R^{d\times d}$ exists and we obtain
\begin{description}[style=unboxed,leftmargin=0cm]
\item[vector-valued ASIP~\cite{MN09,Korepanovapp}]
There exists $\lambda\in(0,\frac12)$ and a probability space $\Omega$ supporting a sequence of random variables $\{S_n,\,n\ge1\}$ with the same joint distributions
as \mbox{$\{v_n,\,n\ge1\}$},
and a sequence $\{Z_n,\,n\ge1\}$ of i.i.d.\ random variables with distribution $N(0,\Sigma)$, such that 
\[
\textstyle \sup_{1\le k\le n} \big|S_k-\sum_{j=1}^k Z_j\big|=
O(n^{\lambda})
\;\text{a.e.\ as $n\to\infty$.}
\]
\end{description}
\end{thm}

\begin{proof}  
The strategy is to model $\bF:\bY\to\bY$ and $F:Y\to Y$ by ``one-sided'' and ``two-sided'' Young towers $\bar\Delta$ and $\Delta$, and to construct an observable $\bar v:\bar\Delta\to\R$ to which the various results in the references can be applied.  The desired statistical properties for $v$ are deduced from those for $\bar v$.

Using $\bF:\bY\to\bY$ and $\rho:\bY\to\Z^+$ as given in Lemma~\ref{lem:AFLV}, we 
define the one-sided Young tower map
$\bar f_\Delta:\bar\Delta\to\bar\Delta$,
\[
\bar\Delta=\{(y,\ell)\in \bY\times\Z^+:0\le\ell\le \rho(y)-1\},
\quad
\bar f_\Delta(y,\ell)=\begin{cases} (y,\ell+1) & \ell\le\rho(y)-2
\\ (\bF y,0) & \ell=\rho(y)-1 \end{cases}.
\]
Let $\bar\mu_Y$ denote the unique absolutely continuous invariant probability measure for the Gibbs-Markov map $\bF:\bY\to\bY$.
Then
$\bar\mu_\Delta=\bar\mu_Y\times{\rm counting}/\int_{\bar
  Y}\rho\,d\bar\mu_Y$
is an ergodic $\bar f_\Delta$-invariant probability measure on $\bar\Delta$.  

Next, define $Y=\pi^{-1}\bY\subset X$ to be the union of stable leaves $W^s_y(X)$ where $y\in \bY$.
In the proof of Corollary~\ref{cor:spectral}, we used
an argument from~\cite{APPV09} which constructs~$\mu$ on $X$ starting from $\bar\mu$ on $\bX$.  
The same argument constructs an ergodic $F$-invariant probability measure $\mu_Y$ on $Y$ starting from $\bar\mu_Y$.  
Define $\rho:Y\to\Z^+$ and $F:Y\to Y$ by setting $\rho(y)=\rho(\pi y)$
and $F(y)=f^{\rho(y)}y$.
Using these (instead of $\rho:\bY\to\Z^+$ and $\bF:\bY\to \bY$)
we obtain a two-sided Young tower map $f_\Delta:\Delta\to \Delta$ 
with ergodic $f_\Delta$-invariant probability measure 
$\mu_\Delta=\mu_Y\times{\rm counting}/\int_Y\rho\,d\mu_Y$.
The projection $\pi:X\to\bX$ extends to a semiconjugacy $\pi:\Delta\to\bar\Delta$ given by $\pi(y,\ell)=(\pi y,\ell)$, and $\pi_*\mu_\Delta=\bar\mu_\Delta$.
Moreover, the projection
\[
\pi_\Delta:\Delta\to X, \qquad
\pi_\Delta(y,\ell)= f^\ell y,
\]
is a semiconjugacy from $f_\Delta$ to $f$ and 
$\pi_{\Delta\,*}\mu_\Delta=\mu$.

The separation time $s(y,y')$ of points $y,y'\in \bY$ is
the least integer $n\ge0$ such that $\bF^ny$ and
$\bF^ny'$ lie in distinct elements of the partition
$\alpha$.  This extends to $\bar\Delta$ by setting $s((y,\ell),(y',\ell'))=s(y,y')$ when $\ell=\ell'$ and zero otherwise, and then to $\Delta$ by setting
$s(p,p')=s(\pi p,\pi p')$.

For each $\theta\in(0,1)$, define the symbolic
metric $d_\theta$ on $\bar\Delta$  given by
$d_\theta(p,p')=\theta^{s(p,p')}$.  
Given $w:\bar\Delta\to\R$, we define
\[
\|w\|_\theta=|w|_\infty+ \sup_{p\neq p'} |w(p)-w(p')|/d_\theta(p, p').
\]

Let $\lambda_1\in(0,1)$ be as in Propositions~\ref{prop:secUH} and~\ref{prop:sec-cone}, and
set $\theta=\lambda_1^{\eta/2}$.
Let $v\in C^\eta(X,\R^d)$ with $\int_X v\,d\mu=0$.   We claim that there exists
$\chi\in L^\infty(\Delta,\R^d)$ and $\bar v\in L^\infty(\bar\Delta,\R^d)$ with
$\|\bar v\|_\theta<\infty$
such that 
\begin{align} \label{eq:chi}
v\circ \pi_\Delta=\bar v\circ\pi+\chi\circ f_\Delta-\chi.
\end{align}
Suppose that the claim is true.
Since $\bar\Delta$ is a one-sided Young tower~\cite{Young99} with superpolynomial tails (in fact $\beta>2$ suffices here) and $\|\bar v\|_\theta<\infty$, 
it follows that $\bar v$ satisfies all of the 
desired statistical properties by the mentioned references.
 These are inherited  (since $\pi$ is measure-preserving) by 
$\bar v\circ\pi:\Delta\to\R^d$.  Since $\chi\in L^\infty$, the properties are
inherited by $v\circ\pi_\Delta:\Delta\to\R^d$ and thereby $v$ (since $\pi_\Delta$ is measure-preserving).

It remains to verify the claim.
Define $\chi:\Delta\to\R^d$,
\[
\chi(p)=\sum_{j=0}^\infty \big(v\circ f^j\circ \pi_\Delta(\pi p)-
v\circ f^j\circ \pi_\Delta(p)\big).
\]
For $p=(y,\ell)$, using Proposition~\ref{prop:secUH}(a), we have
\begin{align*}
|\chi(p)| & \le
\sum_{j=0}^\infty {|v|}_{C^\eta} \|f^j\circ\pi_\Delta(\pi p)-f^j\circ\pi_\Delta(p)\|^\eta
={|v|}_{C^\eta}\sum_{j=0}^\infty \|f^{j+\ell}(\pi y)-f^{j+\ell}(y)\|^\eta
\\ & \le 
{|v|}_{C^\eta}\sum_{j=0}^\infty \lambda_1^{\eta j}\|\pi y-y\|^\eta<\infty.
\end{align*}
Hence $\chi\in L^\infty(\Delta)$.   

Let
$\hat v= v\circ\pi_\Delta-\chi\circ f_\Delta+\chi$.
It follows from the definitions that
$\hat v:\Delta\to\R^d$ is constant along fibres $\pi^{-1}\bar p$ for $\bar p\in\bar\Delta$.
Indeed,
\[
\hat v(p)=\sum_{j=0}^\infty v\circ f^j\circ \pi_\Delta(\pi p)
-\sum_{j=0}^\infty v\circ f^j\circ\pi_\Delta(\pi f_\Delta p).
\]
Hence we can write $\hat v=\bar v\circ\pi$ where
$\bar v:\bar\Delta\to\R^d$ satisfies~\eqref{eq:chi}.

Clearly, $|\bar v|_\infty \le |v|_\infty+2|\chi|_\infty<\infty$.

Let $p=(y,\ell),\,p'=(y',\ell')\in\Delta$.
If $\ell\neq\ell'$, then
$|\bar v(p)-\bar v(p')|\le 2|\bar v|_\infty=2|\bar v|_\infty d_\theta(p,p')$.
When $\ell=\ell'$,
 set $N=[s(p,p')/2]$.  Then
\[
|\hat v(p)-\hat v(p')|\le A_N(p)+A_N(p')+B_N(p,p')+B_{N-1}(f_\Delta p,f_\Delta p'),
\]
where 
\begin{align*}
A_N(q) & =\sum_{j=N}^\infty |v\circ f^j\circ \pi_\Delta (\pi q)
- v\circ f^{j-1}\circ \pi_\Delta(\pi f_\Delta  q)|, \\
B_N(q,q') & =\sum_{j=0}^{N-1} |v\circ f^j\circ \pi_\Delta (\pi q)
- v\circ f^j\circ \pi_\Delta(\pi q')|. 
\end{align*}
The calculation for $\chi$ gives $A_N(q)=O( \lambda_1^{\eta N})=O( \theta^{s(p,p')})$ for $q=p,p'$.
Next, 
\[
B_N(p,p')  =\sum_{j=0}^{N-1} |v\circ f^{j+\ell}(\pi y)
- v\circ f^{j+\ell}(\pi y')|. 
\]
Write $n=s(p,p')$.  By Proposition~\ref{prop:sec-cone},
\begin{align*}
\diam X & \ge \|f^n\circ f^\ell(\pi y)-f^n\circ f^\ell(\pi y')\|
=\|f^{n-j}\circ f^j(f^\ell\pi y)-f^n\circ f^j(f^\ell\pi y')\|
\\ & \ge \lambda_1^{-(n-j)}\|f^j(f^\ell\pi y)-f^j(f^\ell\pi y')\|,
\end{align*}
for all $j\le n$.
Hence
\begin{align} \label{eq:s}
\|f^j(f^\ell \pi y)-f^j(f^\ell \pi y')\|=O( \lambda_1^{s(y,y')-j}),
\end{align}
and so
$B_N(p,p')  \le C\sum_{j=0}^{N-1} \lambda_1^{\eta(s(y,y')-j)}=O(
\theta^{s(p,p')})$.
Similarly, $B_{N-1}(f_\Delta p,f_\Delta p')=O( \theta^{s(p,p')})$.

Hence we have shown that $|\hat v(p)-\hat v(p')|=O( \theta^{s(p,p')})$ and
so $\|\bar v\|_\theta<\infty$ as claimed.
\end{proof}

\begin{rmk}
The ASIP and vector-valued ASIP have numerous consequences summarised in
\cite[p.~233]{MorrowPhilipp82}.  These include
 the central limit theorem (CLT); the functional CLT, also known as the weak invariance principle; the (functional, vector-valued) law of the iterated logarithm (LIL); upper and lower class refinements of the LIL and Chung's LIL.
\end{rmk}

\begin{rmk}  The nondegeneracy assumption $\sigma^2>0$ fails only on a closed subspace of infinite codimension in the space of $C^\eta$ observables.  Indeed if $\sigma^2=0$ and $x\in X$ is a periodic point, then there exists $N\ge1$ such that $\sum_{j=0}^{N-1}v(f^jx)=0$ for all $v\in C^\eta(X)$ with mean zero.
(See~\cite[Theorem~B]{AMV15} for such a result in a more difficult context.)
Similar comments apply to the covariance matrix $\Sigma$ in the vector-valued ASIP.  
Taking one-dimensional projections, we obtain that
the nondegeneracy assumption $\det\Sigma>0$ fails only on a closed subspace of infinite codimension.
\end{rmk}

\section{Statistical properties of singular hyperbolic attractors}
\label{sec:SH}

In this section, we investigate statistical properties of the flow $Z_t$ on a codimension two singular hyperbolic attracting set.
We begin by modifying the Poincar\'e section so that the roof function $\tau$ becomes constant along stable leaves.

Let $\bX$ be the union of $u$-curves in Section~\ref{sec:barf}
and define $X_+=\bigcup_{x\in\bX}W^s_x$.
Then $X_+$ is a H\"older-embedded cross-section and we obtain a new
Poincar\'e map $f_+:X_+\to X_+$ with return time function
$\tau_+:X_+\to \R^+$.
We also define the quotient map $\bar f_+=h\circ f_+:\bX\to\bX$ where $h$ is the stable holonomy in $X_+$.

\begin{prop} \label{prop:tau'}
$\tau_+$ is constant along stable leaves in $\cW^s$ and
$\bar f_+=\bar f$.
\end{prop}

\begin{proof}
For fixed $x\in\bX$, set $T_0=\tau(x)$.  The stable foliation
$\cW^s$ is invariant under the time $T_0$-map $Z_{T_0}$ so
$Z_{T_0}(W^s_x)=W^s_{T_0x}\subset X_+$.
Hence $\tau_+(x)=T_0$ for each $x\in W^s_x$.

Next, recall that $W^s_{fx}(X)$ is the intersection
of $\bigcup_{|t|<\eps_0}Z_t W^s_x$ with $X$ for suitably chosen $\eps_0$.  
Then $\bar f x$ is the unique intersection point of $\bigcup_{|t|<\eps_0}Z_t W^s_x$ with $\bX$.  
But $f_+ x=Z_t fx$ for some small $t$ so $\bar f_+ x$ also lies in the intersection of $\bigcup_{|t|<\eps_0}Z_t W^s_x$ with $\bX$.  Hence $\bar f_+ x=\bar f x$.
\end{proof}

In this section, we work with the new Poincar\'e map and
roof function which we relabel $f:X\to X$ and $\tau:X\to\R^+$.
In doing so we lose the smoothness properties of $f$ and $\tau$ --- they are now only piecewise H\"older.  However we gain the property that
$\tau$ is constant along the stable foliation in $X$.
Since $\bar f:\bX\to\bX$ is unchanged; we still have that $\bar f$ is piecewise $C^{1+\eps}$ and the results on $\bar f$ in Section~\ref{sec:barf} and the physical measures and statistical properties in Section~\ref{sec:stat} remain valid.

Define the suspension 
\[
X^\tau=\{(x,u)\in X\times\R: 0\le u\le\tau(x)\}/\sim\quad\text{where $(x,\tau(x))\sim(fx,0)$},
\]
and the suspension flow $(x,u)\mapsto(x,u+t)$ (computed modulo identifications).

\begin{thm} \label{thm:spectral}  There exists a 
finite number of 
ergodic $Z_t$-invariant probability measures $\mu_{M,1},\dots,\mu_{M,s}$ whose basins cover a subset of $U_0$ of full Lebesgue measure.   
\end{thm}

\begin{proof}
For each $\mu_j$ in Corollary~\ref{cor:spectral}, we obtain an ergodic flow-invariant probability measure $\mu_j^\tau=\mu_j\times{\rm Lebesgue}/\int_X\tau\,d\mu_j$
on $X^\tau$.  The projection $\pi^\tau:X^\tau\to M$, $\pi^\tau(x,u)=Z_ux$ defines a semiconjugacy from $X^\tau$ to $M$ and 
$\mu_{M,j}=\pi^\tau_*\mu_j^\tau$ is an ergodic $Z_t$-invariant probability measure on $M$.
By~\cite[Section~7]{APPV09}, these form a finite family of physical measures
$\mu_{M,j}$ for the flow $Z_t$ whose basins cover a subset of
$U_0$ of full Lebesgue measure.
\end{proof}

Suppose without loss that there is a unique physical measure $\mu_M=\pi^\tau_*\mu^\tau$ where $\mu^\tau=\mu\times{\rm Lebesgue}/\int_X\tau\,d\mu$ (in the notation above).
Recall that $\bar\mu$, and hence $\mu$, is mixing up to a finite cycle of length $k\ge1$.
By shrinking the cross-section $X$ we may suppose without loss that the measure $\mu$ on $X$ is mixing.

Define the {\em induced} roof function
\[
\varphi:\bY\to\R^+,\qquad  \varphi(y)=
\textstyle{\sum_{\ell=0}^{\rho(y)-1}}\tau(\bar f^\ell y).
\]

\begin{prop} \label{prop:tail}
$\mu_Y(\varphi>t)=O(t^{-\beta})$ for any $\beta>0$.
\end{prop}

\begin{proof}
A standard general calculation (see for example~\cite[Proposition~A.1]{BMprep}) shows that
\[
\mu_Y(\varphi>t)\le \mu_Y(\rho>k)+\bar\rho \mu(\tau>t/k),
\]
for all $t>0$, $k\ge1$, where $\bar\rho=\int_Y\rho\,d\mu_Y$.
In particular, 
since $\rho$ has superpolynomial tails and $\tau$ has at most logarithmic singularities, there is a constant $c>0$ such that
$\mu_Y(\varphi>t)=O( k^{-2\beta}+ e^{-ct/k})$.
Now take $k=[t^{1/2}]$.
\end{proof}

Recall that $\bF:\bY\to\bY$ is a Gibbs-Markov map with partition $\alpha$ and separation time $s(y,y')$.   

\begin{prop} \label{prop:varphi}
There exists $\theta\in(0,1)$ and $C>0$ such that
\[
|\varphi(y)-\varphi(y')|\le C\theta^{s(y,y')}
\quad \text{for all $y,y'\in a$, $a\in\alpha$.}
\]
\end{prop}

\begin{proof}
We can write $\tau=\tau_0+\tau_1$ where $\tau_0$ is as in previous sections and in particular satisfies the estimate 
Proposition~\ref{prop:tau}, and $\tau_1$ is $C^\eps$.
Setting $\theta=2^{-\eps}$ and using uniform expansion of $\bar f$,
\[
|\tau_1(\bar f^\ell y)-\tau_1(\bar f^\ell y')|
\le {|\tau_1|}_{C^\eps}|\bar f^\ell y-\bar f^\ell y'|^\eps
\le C_1 
\theta^{\rho(y)-\ell}|\bF y-\bF y'|^\eps.
\]
Combining this with the estimate for $\tau_0$, we obtain that 
$\sum_{\ell=0}^{\rho(y)-1}|\tau(\bar f^\ell y)-\tau(\bar f^\ell y')|
\le C_1' |\bF y-\bF y'|^\eps$.
By~\eqref{eq:s}, $|\bF y-\bF y'|=O(2^{-s(y,y')})$ and the result follows.
\end{proof}

\subsection{Statistical limit laws for the flow}
\label{sec:statflow}

If $\Lambda=\supp\mu_M$ contains no equilibria, then $\Lambda$ is a nontrivial hyperbolic basic set for an Axiom~A flow and the CLT for H\"older observables
follows from~\cite{Ratner73,MT04}.  Moreover,~\cite{DenkerPhilipp84} obtains a version of the (scalar) ASIP that implies the functional CLT and functional LIL.

When $\Lambda$ contains equilibria, the CLT and its functional version still holds by~\cite{HM07} at least for geometric Lorenz attractors.
As pointed out in~\cite{BMprep}, a simpler argument than in~\cite{HM07} applies in general situations where the roof function is unbounded
and includes the entire class of singular hyperbolic attractors analysed in this paper.  We refer to the introduction of~\cite{BMprep} for a more comprehensive list of statistical limit laws, with precise statements, that can be obtained in this way.
%

\subsection{Mixing and superpolynomial mixing for the flow}

\begin{thm} \label{thm:super}
There is a $C^2$-open and $C^\infty$-dense set of singular hyperbolic flows such that each nontrivial attractor $\Lambda$ is mixing with superpolynomial decay of correlations: for any $\beta>0$,
\[
\Big|\int_\Lambda v\,w\circ Z_t\,d\mu_M-
\int_\Lambda v\,d\mu_M
\int_\Lambda w\,d\mu_M\Big|\le Ct^{-\beta}\quad\text{for all $t>0$}, 
\]
for all $v,w:M\to\R$ such that one of $v$ or $w$ is
$C^\infty$ and the other is H\"older.  Here $C$ is a
constant depending on $v$, $w$ and $\beta$.
\end{thm}

\begin{proof}  
If $\Lambda=\supp \mu_M$ contains no equilibria, then $\Lambda$ is uniformly hyperbolic and the result is due to~\cite{Dolgopyat98a,FMT07}.
The general case follows essentially from~\cite{M07,M09}.

More precisely,
we have seen that the semiflow and flow is modelled as a suspension over a Young tower with superpolynomial tails.  Using the induced roof function
$\varphi:Y\to\R^+$, 
we obtain a suspension $Y^\varphi$ over the uniformly hyperbolic map
$F:Y\to Y$ where the roof function $\varphi:Y\to\R^+$ has superpolynomial tails.

We are now in a position to apply~\cite[Theorem~3.1]{BBMsub}
(see also~\cite[Theorem~4.1]{rapid}).
Conditions~(3.1) and~(3.2) in~\cite{BBMsub} follow from Propositions~\ref{prop:secUH} and~\ref{prop:sec-cone}.
Moreover, $\varphi$ is constant along stable leaves by Proposition~\ref{prop:tau'} and projects to a well-defined roof function
$\varphi:\bY\to\R^+$ satisfying the estimate in Proposition~\ref{prop:varphi} which is condition~(3.3) in~\cite{BBMsub}.  Hence the suspension flow on $Y^\varphi$ is a skew product Gibbs-Markov flow in the terminology of~\cite{BBMsub}.  
Hence
superpolynomial mixing follows from~\cite[Theorem~3.1]{BBMsub}
subject to a nondegeneracy condition (absence of approximate eigenfunctions).

Finally, it is shown in~\cite{FMT07} that absence of approximate eigenfunctions is $C^2$-open and $C^\infty$-dense (cf.~\cite[Remark~2.5]{M09} or~\cite[Subsection~5.2]{rapid}).
\end{proof}

We have already seen that statistical limit laws such as the CLT hold for all singular hyperbolic flows.  In the situation of Theorem~\ref{thm:super}, we can obtain such results also for the time-one map of a singular hyperbolic flow.

\begin{cor} \label{cor:super}
Assume that $Z_t:\Lambda\to\Lambda$ has superpolynomial decay of correlations as in Theorem~\ref{thm:super}.
Let $v:M\to\R$ be $C^\infty$ (or at least sufficiently smooth) with mean zero.
Then the ASIP holds for the time-one map $Z_1$ for all
$C^\infty$ observables $v:M\to\R$.

In particular, the limit
$\sigma^2=\lim_{n\to\infty}n^{-1}\int_\Lambda(\sum_{j=0}^{n-1}v\circ Z_j)^2\,d\mu_M$ exists, and after passing to an enriched probability space, there exists a sequence $A_0,A_1,\ldots$ of i.i.d.\ normal random variables with mean zero and variance $\sigma^2$ such that
\[
\sum_{j=0}^{n-1}v\circ Z_j=\sum_{j=0}^{n-1}A_j
+O(n^{1/4}(\log n)^{1/2}(\log\log n)^{1/4}), \quad a.e.
\]
Moreover, if $\sigma^2=0$, then for every periodic point $q\in\Lambda$, there exists $T>0$ (independent of $v$) such that $\int_0^Tv(Z_tq)\,dt=0$.
\end{cor}

\begin{proof}
This is proved in the same way as~\cite[Theorems~B and C]{AMV15}.
\end{proof}

In the case of the classical Lorenz attractor, it was shown in~\cite{LMP05} and~\cite{AMV15} that mixing and superpolynomial mixing is automatic.  The proof exploits the {\em locally eventually onto (l.e.o.)}~property as well as smoothness properties of the stable foliation.  
We now show that the mixing argument in~\cite{LMP05} does not require the stable foliation to be smooth.
In the general situation of this paper, we assume hypotheses that are more complicated to state but which are implied by l.e.o.\
for the classical Lorenz attractor.

We require that $\Lambda$ contains at least one equilibrium.
Let $q\in\cS$ be the corresponding singularity for $\bar f:\bX\to\bX$.
(Again, $\bar f$ is not defined at $q$.)
Assume that the set of preimages of $q$ under iterates of $\bar f$ is 
dense in $\bX$.
(This condition is always satisfied for geometric Lorenz attractors.)
By Lemma~\ref{lem:AFLV} and Remark~\ref{rmk:ALP}, we can construct an induced Gibbs-Markov map $\bF=\bar f^\sigma:\bY\to\bY$ where the inducing set $\bY$ contains~$q$.  Let $K=\bigcup_{\ell\ge 0}\bar f^\ell\bY$; this is an open and dense full measure subset of $\bX$.   Our final assumption is
that 
$\bar f^pq_+=\lim_{y\to q+}\bar f^py\in K$ for some $p\ge1$.  (This would work equally well with $q_+$ replaced by $q_-$.)

\begin{thm}  \label{thm:mix}
Under the above assumptions,
$\Lambda$ is automatically mixing (and even Bernoulli).
\end{thm}

\begin{proof}  We sketch the proof following~\cite{LMP05}.
By~\cite{Ratner78}, it suffices to show that the quotient suspension semiflow
$\bar f_t^\tau:\bX^\tau\to\bX^\tau$ is weak mixing.  Equivalently,
the cohomological equation $u\circ \bar f=e^{ib\tau}u$ has no measurable solutions $u:\bX\to S^1$ for all $b\neq0$.
(Here $S^1$ denotes the unit circle in $\C$.)

Suppose for contradiction that there exists $u:\bX\to S^1$ measurable and
$b\neq0$ such that $u\circ \bar f=e^{ib\tau}u$.
A Liv\v{s}ic regularity theorem
of~\cite{BruinHollandNicol05}, exploiting the fact that 
$\bF$ is Gibbs-Markov and that
the roof function $\tau$ is H\"older with
at most logarithmic growth (Lemma~\ref{lem:log}) ensures
that $u$ has a version that is continuous on $K$.

Also, $q\in \bY\subset K$.  Choose $p\ge1$ with
$\bar f^pq_+\in K$.  Then
$u\circ \bar f^p=e^{ib\tau_p}u$ where
$\tau_p=\sum_{j=0}^{p-1}\tau\circ \bar f^j$.  By
Remark~\ref{rmk:log}, $\tau_p(y)\ge\tau(y)\to\infty$ as
$y\to q_+$, whereas $u(y)\to u(q)$ and
$u(\bar f^py)\to u(\bar f^pq_+)$.  Since $b\neq0$, this
contradicts the equality $u\circ \bar f^p=e^{ib\tau_p}u$.
\end{proof}

\begin{rmk} If we assume in addition that the
  stable foliation $\cW^s$ for the flow is $C^{1+\eps}$,
  then we can deduce exponential decay of correlations
  following~\cite{AraujoM16}.  

However, without smoothness
  of $\cW^s$, the roof function $\tau$ (on the modified
  cross-section) is only H\"older and the cancellation
  argument of~\cite{Dolgopyat98a} fails.  In fact, we are
  unable even to prove superpolynomial mixing for fixed
  flows (without perturbing as in Theorem~\ref{thm:super}).
It should be possible to use the
  techniques in~\cite{AMV15} to prove that the stable and
  unstable foliations (defined appropriately) for the flow
  are not jointly integrable -- this is a stronger property
  than mixing.  However, we do not see how to use this to
  prove superpolynomial mixing when $\tau$ is only H\"older.
\end{rmk}


\appendix

\section{Theorem of Alves {\em et al.}~\cite{AFLV11}}

In this appendix, we recall a result of  Alves {\em et al.}~\cite{AFLV11} that is required in Section~\ref{sec:spectral}.   
Although the argument in~\cite{AFLV11} is essentially correct, there are certain problems with the formulation of the hypotheses.  First, 
the hypotheses (C2) and (C3) in~\cite{AFLV11} are stated too strongly, since 
the right-hand side of their conditions are zero for
points $x\neq y$ equidistant from $\cS$, whereas the left-hand side is generally nonzero.   Second, the hypotheses are not stated strongly enough for 
the first half of the proof of~\cite[Lemma~5.1]{AFLV11}, since the estimate for $d(x,\cS)^{-\alpha}$ is false in general.
We state below a corrected version of the hypotheses in~\cite{AFLV11}.  The conclusion in Theorem~\ref{thm:AFLV} is identical to that in~\cite[Theorem~C]{AFLV11}, and the proof is largely unchanged.

Throughout, $(M,d)$ is a compact Riemannian manifold and
$f:M\to M$ is a local $C^{1+}$ diffeomorphism with singularity set $\cS$.
We suppose that there are constants $\eta\in(0,1)$ and $C,\,q>0$ such that
\begin{itemize}
\item[(C0)] $\Leb(x:d(x,\cS)\le\eps)\le C\eps^\eta$ for all $\eps\ge0$.
\item[(C1)]  $C^{-1}d(x,\cS)^q\le \|Df(x)v\|\le Cd(x,\cS)^{-q}$,
for all $x\in M\setminus\cS$, $v\in T_xM$ with $\|v\|=1$.
\item[(C2)] $\big|\log\|Df(x)^{-1}\|-\log\|Df(y)^{-1}\|\big|\le Cd(x,y)^\eta(\|Df(x)^{-1}\|^q+\|Df(x)^{-1}\|^{-q})$ for all $x,y\in M\setminus\cS$
with $d(x,y)<d(x,\cS)/2$.
\item[(C3)] $\big|\log|\det Df(x)|-\log|\det Df(y)| \big|\le C d(x,y)^\eta d(x,\cS)^{-q}$ for all $x,y\in M\setminus \cS$
with $d(x,y)<d(x,\cS)/2$.
\end{itemize}

Recall~\cite[Definition~1.2]{AFLV11} that a measure $\mu$ is expanding
if $\log\|(Df)^{-1}\|$ is integrable with respect to $\mu$ and
$\int_M\log\|(Df)^{-1}\|\,d\mu<0$.

Let $\Cov(v,w)= \int_M v\,w\,d\mu-\int_M v\,d\mu\int_M w\,d\mu $.

\begin{thm}[ {\cite[Theorem~C]{AFLV11}} ] \label{thm:AFLV}
Let $f:M\to M$ be a $C^{1+}$ local diffeomorphism satisfying (C0)--(C3), and let $\alpha\in(0,1)$.
Let $\mu$ be an ergodic expanding absolutely continuous invariant probability measure with $d\mu/d\Leb\in L^p$ for some $p>1$.

\vspace{1ex}
\noindent
(1) Suppose that there exists $\beta>1$ and $C>0$  such that 
$|\Cov(v,w\circ f^n)|\le C\|v\|_{C^\alpha} |w|_\infty\,n^{-\beta}$
for all $v\in C^\alpha$, $w\in L^\infty$, $n\ge1$.

Then  there is a full branch Gibbs-Markov induced map $F=f^\rho:Y\to Y$, where $\rho:Y\to\Z^+$ is constant on partition elements and satisfies $\Leb(y\in Y:\rho(y)>n)=O(n^{-(\beta-1)})$.
Moreover, there are constants $C,\,\eps>0$ such that 
\begin{align} \label{eq:hyptime}
\textstyle{\sum_{\ell=0}^{\rho(y)-1}}d(f^\ell y,f^\ell y')^\eta d(x,\cS)^{-q}\le Cd(Fy,Fy')^\eps,
\end{align}
for all $y,y'$ lying in the same partition element.

\vspace{1ex}
\noindent
(2)  Suppose that $d\mu/d\Leb$ is bounded below on its support and that there exist $\gamma\in(0,1]$, $C,\,c>0$ such that
$|\Cov(v,w\circ f^n)|\le C\|v\|_{C^\alpha} |w|_\infty\,e^{-cn^\gamma}$
for all $v\in C^\alpha$, $w\in L^\infty$, $n\ge1$.

Then the conclusion in (1) holds and moreover for any $\gamma'\in(0,\gamma/(3\gamma+6))$ there exists $c'>0$ such that
$\Leb(y\in Y:\rho(y)>n)=O\big(e^{-c'n^{\gamma'}}\big)$.
\end{thm}

\begin{rmk} \label{rmk:fF}
The estimate~\eqref{eq:hyptime} is 
a crucial component of the proofs
in~\cite{AFLV11,AlvesLuzzattoPinheiro05}.
 (See the calculation at the end of the proof of~\cite[Lemma~4.1]{AlvesLuzzattoPinheiro05}.)
We make it explicit here since it is used in the proof of Proposition~\ref{prop:tau}.
\end{rmk}

\begin{rmk} \label{rmk:ALP}
Let $x\in M$ be any point with dense preimages in $M$.  By~\cite[Remarks~1.4]{AlvesLuzzattoPinheiro05},
the inducing set $Y$ can be chosen to be an open ball containing $x$.
\end{rmk}

In the remainder of this appendix, we indicate the modifications to the argument in~\cite{AFLV11} required to obtain the corrected version of Theorem~\ref{thm:AFLV}.

We begin by noting that a consequence of (C1) and (C2) is that 
\[
\big|\log\|Df(x)^{-1}\|-\log\|Df(y)^{-1}\| \big|\le C d(x,y)^\eta d(x,\cS)^{-q^2}
\quad\text{for all $x,y\in M\setminus \cS$.}
\]
Combined with (C3),
this means that the $C^{1+}$ version of the $C^2$ set up in~\cite{AlvesBonattiViana00,AlvesLuzzattoPinheiro05} is satisfied.
It is well-known, and routine, that the theory of hyperbolic times and the resulting constructions in~\cite{AlvesBonattiViana00,AlvesLuzzattoPinheiro05} work just as well in the $C^{1+}$ setting.
Hence as in~\cite{AFLV11}, it suffices to 
verify the hypotheses of~\cite[Theorem~2]{AlvesLuzzattoPinheiro05}.
This all proceeds exactly as in~\cite{AFLV11} except for the estimate of $\phi_{1,k}$ in~\cite[Lemma~5.1]{AFLV11}.  
Recall from~\cite{AFLV11} that
$\phi_1=\log\|(Df)^{-1}\|$
and that
$\phi_{1,k}=\phi_1 1_{\{|\phi_1|\le k\}}$.
(The definition in~\cite{AFLV11} has 
$\phi_{1,k}=\phi_1 1_{\{\phi_1\le k\}}$.
but it is clear from the proof of~\cite[Lemma~4.3]{AFLV11} that this is what was meant.)

\begin{prop} \label{prop:AFLV}
For any $\alpha>0$, there exists $\eta'\in(0,1)$, $C>0$ such that
$\|\phi_{1,k}\|_{C^{\eta'}}\le Ce^{\alpha k}$.
\end{prop}

\begin{proof}  
We can suppose without loss that $\alpha<2q$.

Let $x,y\in M$. It is immediate that $|\phi_{1,k}(x)|\le k$ and that
$|\phi_{1,k}(x)-\phi_{1,k}(y)|\le 2k$.
Also, by (C2), assuming without loss that $\phi_{1,k}(x)\le \phi_{1,k}(y)$, 
\[
|\phi_{1,k}(x)-\phi_{1,k}(y)| \le C_1 d(x,y)^\eta (e^{q\phi_{1,k}(y)}+e^{-q\phi_{1,k}(x)})
\le C_1' d(x,y)^\eta e^{q k}.
\]
The inequality $\min\{1,a\}\le a^\eps$ holds for all $a\ge0$, $\eps\in[0,1]$.
 Hence 
taking $\eps=\frac12\alpha/q$ and $\eta'=\eps\eta$ we obtain that
\[
|\phi_{1,k}(x)-\phi_{1,k}(y)| 
\le C_2 k\min\{1,d(x,y)^\eta e^{q k}\}
\le C_2k
d(x,y)^{\eta'} e^{\frac12\alpha k}\le C_2' d(x,y)^{\eta'} e^{\alpha k}.
\]
We have shown that $\|\phi_{1,k}\|_{C^{\eta'}}=O( k+e^{\alpha k})
=O( e^{\alpha k})$ as required.
\end{proof}

The remainder of the proof of Theorem~\ref{thm:AFLV} proceeds exactly as in~\cite{AFLV11}.  (We note that in~\cite{AFLV11} it is asserted that $\eta'=\alpha$, but this is not required in the proof.)


 \def\cprime{$'$}

\def\polhk#1{\setbox0=\hbox{#1}{\ooalign{\hidewidth
  \lower1.5ex\hbox{`}\hidewidth\crcr\unhbox0}}}

\end{document}